\documentclass[12pt]{amsart}
\usepackage[shadow]{todonotes}
\usepackage{geometry}
\usepackage[utf8]{inputenc}
\usepackage[english]{babel}
\usepackage{amsmath,amsfonts,amssymb,amsthm}
\usepackage{bbm}
\usepackage{hyperref}
\usepackage{verbatim}
\usepackage{tikz}
\usetikzlibrary{cd}
\usepackage{mathdots,mathrsfs,microtype,mathtools}

\theoremstyle{plain}
\newtheorem{teo}{Theorem}[section]
\newtheorem{cor}[teo]{Corollary}
\newtheorem{prop}[teo]{Proposition}
\newtheorem{lem}[teo]{Lemma}
\newtheorem{claim}[teo]{Claim}

\theoremstyle{definition}
\newtheorem{defin}[teo]{Definition}
\theoremstyle{remark}

\newcommand{\system}[1]{\mbox{\fontfamily{cmss}\fontshape{n}\fontseries{m}\selectfont#1}}

\newcommand{\ZFC}{\system{ZFC}}
\newcommand{\AC}{\system{AC}}

\DeclareMathOperator{\Fml}{Fml}
\DeclareMathOperator{\cof}{cof}
\DeclareMathOperator{\crt}{crt}

\DeclareMathOperator{\Ord}{Ord}
\DeclareMathOperator{\Ult}{Ult}

\DeclareMathOperator{\ran}{ran}

\DeclareMathOperator{\Ithree}{\mathsf{I3}}
\DeclareMathOperator{\Itwo}{\mathsf{I2}}
\DeclareMathOperator{\Ione}{\mathsf{I1}}
\DeclareMathOperator{\Izero}{\mathsf{I0}}

\date{\today} 

\begin{document}
\title{The iterability hierarchy above $\Ithree$}
\author{Alessandro Andretta}
\address{Dipartimento di Matematica, Università di Torino, via Carlo Alberto 10, 10123 Torino---Italy}
\email{alessandro.andretta@unito.it}
\author{Vincenzo Dimonte}
\address{Università degli Studi di Udine, Dipartimento di Scienze Matematiche, Informatiche e Fisiche, via delle Scienze, 206, 33100 Udine---Italy}
\email{vincenzo.dimonte@gmail.com}
\urladdr{https://users.dimi.uniud.it/~vincenzo.dimonte/}


%

\begin{abstract}
In this paper we introduce a new hierarchy of large cardinals between $\Ithree$ and $\Itwo$, the iterability hierarchy, and we prove that every step of it strongly implies the ones below.
\end{abstract}
\maketitle

\section{Introduction}

In the late 70's and early 80's there was a flurry of activity around rank-into-rank axioms, a new kind of large cardinal hypotheses at the top of the hierarchy. 
They were $\Ithree$ (the existence of an elementary embedding $j \colon V_\lambda\prec V_\lambda$), $\Itwo$ (the existence of an elementary embedding $j \colon V\prec M$ with $V_\lambda\subseteq M$) and $\Ione$ (the existence of an elementary embedding $j \colon V_{\lambda+1}\prec V_{\lambda+1}$). 
The initial belief of the majority of set theorists was that their existence was eventually going to be disproved in \( \ZFC \), and hence the $\mathsf{I}$ suggesting inconsistency, but a result by Martin started to change the mood: he proved in~\cite{Martin} that a hypothesis strictly below $\Itwo$ and above $\Ithree$ implies the determinacy of $\boldsymbol{\Pi}^1_2$ sets. 
The hypothesis used by Martin was the existence of an \emph{iterable} (see Section~\ref{sec:iterate}) \( j \colon V_  \lambda \prec V_  \lambda \)---we call this large cardinal axiom \( \Ithree_\infty \).
The excitement grew when Woodin, a few years later, proved the consistency of the Axiom of Determinacy using an axiom stronger than $\Ione$, called $\Izero$. 
But in the following years Woodin, building on work of Martin and Steel, showed that determinacy had much lower consistency strength, so \( \Ithree_\infty \) slowly fell into oblivion. 
Interest in rank-into-rank axioms re-emerged twenty years ago, thanks to Laver's results on the algebra of the $\Ithree$-embeddings, and more recently after Woodin's extensive work on $\Izero$. 
Moving to the present age, the researcher that wants to know more about \( \Ithree_\infty \), however, is going to be disappointed: the only reference is Martin's original paper, that is very terse and lacking in details. 
For example, it says that $\Itwo$ strictly implies \( \Ithree_\infty \), but it provides no proof for that. 
Even the very definition of iterable embedding is not fully satisfactory, as it is founded on operations whose validity has not been fully provided in print (as in Lemma~\ref{lem:application}). 
Also, is \( \Ithree_\infty \) strictly stronger than $\Ithree$?

The aim of this paper is to approach iterability of $\Ithree$-embeddings in all the details and in a modern way, thanks to the better understanding of rank-into-rank axioms that decades of work has given us. 
In Section~\ref{pre} the current knowledge of rank-into-rank axioms is described, introducing the concept of ``strong implication'', more suitable for such axioms than the usual notion of strict implication. 
In Section~\ref{sec:iterate} we define exactly what it is an iterable $\Ithree$-embedding, and we introduce a new hierarchy of axioms, depending on how long the $\Ithree$-embedding can be iterable. 
It turns out that, like for iterations of a single measure or an extender, if an $\Ithree$-embedding is $\omega_1$-iterable, then it is iterable for any length, and this is strictly (in fact: strongly) weaker than $\Itwo$. 
Finally, in Section~\ref{hierarchy} we prove that every step in the iterability hierarchy strongly implies the ones below, with even one more step at limit points. 
The final picture is therefore the following:
\[
\begin{tikzpicture}
 \node (I2){$\Itwo $};
 \node (iterable)[below of = I2]{$\Ithree_\infty$}
	edge[<-](I2);
 \node (almostiterable)[below of = iterable]{$\forall \alpha<\omega_1\ \Ithree_\alpha$}
 edge[<-](iterable);
 \node (dots)[below of = almostiterable]{$\vdots$}
 edge[<-](almostiterable);
 \node (omegaomega)[below of = dots]{$\Ithree_\omega$}
 edge[<-](dots);
 \node (lessomegaomega)[below of = omegaomega]{$\Ithree_{<\omega}$}
 edge[<-](omegaomega);
 \node (almostomegaomega)[below of = lessomegaomega]{$\forall n<\omega \Ithree_n$}
 edge[<-](lessomegaomega);
 \node (dots2)[below of = almostomegaomega]{$\vdots$}
 edge[<-](almostomegaomega);
 \node (omega2)[below of = dots2]{$\Ithree_2$}
 edge[<-](dots2);
 \node (omega1)[below of = omega2]{$\Ithree_1$}
 edge[<-](omega2);
 \node (base)[below of = omega1]{$\Ithree \leftrightarrow \Ithree_0 $}
 edge[<-](omega1);
\end{tikzpicture}
\]
\textbf{Acknowledgments}. The second author would like to thank the program ``Rita Levi Montalcini 2013'' for the support, and the Department of Mathematics of the University of Turin for its hospitality.

\section{Preliminaries}\label{pre}

To avoid confusion or misunderstandings, all notations and standard basic results are collected here.

If $M$ and $N$ are sets or classes, $j \colon M\prec N$ denotes that $j$ is an elementary embedding from $M$ to $N$, that is a function such that for any formula $\varphi$ and any $x\in M$, $M \vDash\varphi ( x )$ iff $N\vDash\varphi ( j (x) )$; when this holds only for $\Sigma_n$ formul\ae, we write $j \colon M\prec_n N$.
The case in which $j$ is the identity, i.e., if $M$ is an elementary (or $\Sigma_n$-elementary) submodel of $N$, is simply written as $M\prec N$ (or $M\prec_n N$).

If $M\vDash\AC$ or $N\subseteq M$ and $j \colon M\prec N$ is not the identity, then it moves at least one ordinal; the least such ordinal is the \emph{critical point} of $j$ and it is denoted by $\crt(j)$.
Let $j$ be an elementary embedding and $\kappa = \crt(j)$. 
Define $\kappa_0=\kappa$ and $\kappa_{n+1}=j(\kappa_n)$. 
Then $\langle \kappa_n : n\in\omega\rangle$ is the \emph{critical sequence} of $j$.
 
Kunen~\cite{Kunen} proved that if $M=N=V_\eta$ for some ordinal $\eta$, and $\lambda$ is the supremum of the critical sequence, then $\eta$ cannot be bigger than $\lambda+1$ (and of course cannot be smaller than $\lambda$).
Kunen's result actually does not say anything about the cases $\eta=\lambda$ or $\eta=\lambda+1$.
Therefore we can introduce the following hypotheses without fearing an immediate inconsistency:

\begin{description}
 \item[$\Ithree$] There exists $j \colon V_\lambda\prec V_\lambda$, where $\lambda$ is the supremum of the critical sequence of $j$.
 \item[$\Ione$] There exists $j \colon V_{\lambda+1}\prec V_{\lambda+1}$, where $\lambda$ is the supremum of the critical sequence of $j$.
\end{description}
We will be flexible in handling this and other rank-into-rank notations, but the meaning will be always clear. 
For example, $\Ithree ( \lambda )$ means that there is a $ j \colon V_\lambda\prec V_\lambda$ (so it is a property of $\lambda$), while $\Ithree ( j ) $ indicates that $ j \colon V_\lambda\prec V_\lambda$.
Sometimes we write $\Ithree(j,\lambda)$ to underline the role of $\lambda$, etc. 

Another way to reach the apogee of the large cardinal hierarchy is via the usual template asserting the existence of an elementary embedding $j \colon V\prec M\subseteq V$ with \( M \) resembling \( V \).

\begin{defin}
A cardinal $\kappa$ is:
\begin{itemize}
\item
\emph{superstrong} iff there exists $j \colon V\prec M$ such that $\crt(j)=\kappa$ and $V_{\kappa_1}\subseteq M$, where $\kappa_1$ is the second element of the critical sequence of $j$;
\item
\emph{$n$-superstrong} iff there exists $j \colon V\prec M$ such that $\crt(j)=\kappa$ and $V_{\kappa_n}\subseteq M$, where $\kappa_n$ is the $n+1$-th element of the critical sequence of $j$;
\item
\emph{$\omega$-superstrong} iff there exists $j \colon V\prec M$ such that $\crt(j) = \kappa$ and $V_{\lambda}\subseteq M$, where $\lambda$ is the supremum of the critical sequence of $j$.
\end{itemize}
\end{defin} 

If $\kappa$ is $\omega$-superstrong as witnessed by $j,\lambda$, then $j(\lambda)=\lambda$, and therefore $\Ithree$ holds. 
Moreover, like other large cardinals, it can be formulated as the existence of an extender (see~\cite{Kanamori}), in this case a $( \kappa , \lambda ) $-extender $E$ such that $ V_\lambda \subseteq \Ult ( V , E ) $, where $\lambda$ is the supremum of the $\kappa_n$'s. 

It is possible to pinpoint exactly how much $\omega$-superstrongness is stronger than $\Ithree$, but for this we need to clarify what we mean by ``being stronger'':

\begin{defin}\label{strongly}
Let $\Phi ( j , \lambda ) $ and $ \Psi ( j , \lambda ) $ be two large cardinal properties as above. 
Then
\begin{itemize}
\item $ \Phi $ implies $\Psi$ iff $\ZFC\vdash \Phi ( j , \lambda ) \rightarrow \Psi ( j , \lambda ) $;
\item $ \Phi $ strictly implies $\Psi$ iff $\Phi$ implies $\Psi$ and $\ZFC\vdash \Phi ( j , \lambda ) \rightarrow \exists \lambda'\ \exists j'\ ( \Psi ( j' , \lambda ' ) \wedge \neg ( \Phi ( j' , \lambda' ) ) ) $;
\item $\Phi$ strongly implies $\Psi $ iff $ \Phi $ implies $ \Psi $ and $\ZFC\vdash \Phi ( j ,\lambda ) \rightarrow\exists \lambda'<\lambda\ \exists j'\ \Psi (j',\lambda')$.
\end{itemize}
\end{defin}

Note that strong implication yields strict implication: If $\Phi$ strongly implies $\Psi$, let $\lambda$ be the smallest such that there is a $j$ such that $ \Phi (j,\lambda)$ holds. 
Then there are $j'$ and $\lambda'<\lambda$ such that $ \Psi (j',\lambda')$ holds, and since $\lambda$ was the smallest for $\Phi$, then $ \Phi (j',\lambda')$ does not hold. 
The difference between ``strict'' versus ``strong'' is a consequence of the peculiar nature of rank-into-rank axioms. 
For weaker axioms, usually the focal cardinal is the critical point of an elementary embedding, and such cardinal is measurable. 
If, assuming some property $\Phi $ of $\kappa$, we can find a $\kappa'<\kappa$ that satisfies a property $\Psi$, the reasoning is as follows: let $\kappa$ be the smallest cardinal that satisfies $ \Phi $; find $\kappa'<\kappa$ that satisfies $\Psi$; then $\kappa'$ must not satisfy $\Phi $; then $V_\kappa$ is a model of \ZFC{} where no cardinal satisfy $ \Phi $, but some cardinal satisfy $ \Psi $, so the consistency strength of $\Phi $ is stronger than that of $\Psi$. 
So actually strict implication and strong implication are the same.

In the rank-into-rank case, the focal cardinal is $\lambda$, as for the same $\lambda$ there can be many different embeddings $j$ and critical points of them (see discussion after Lemma~\ref{lem:operations}). 
Therefore, given an embedding $j \colon V_\lambda\prec V_\lambda$ satisfying a property $\Phi$, if we find another embedding $j' \colon V_\lambda\prec V_\lambda$ with $\crt(j') < \crt(j)$ that satisfies a property $\Psi$, this would prove that the two properties are actually different, so it would prove that $ \Phi $ strictly implies $\Psi$: Let $j,\lambda$ satisfy $ \Phi $ with $\crt(j)$ the smallest possible. 
Then if we can find $j':V_\lambda\prec V_\lambda$ with $\crt(j')<\crt(j)$, and so $ \Phi ( j' , \lambda ) $ does not hold. 
But this falls short of actually proving that the consistency strength of one is stronger than the other. 
As $V_\lambda\vDash\ZFC$, if $\lambda$ is least for $\Phi$ and $\Phi$ strongly implies $ \Psi $, then $V_\lambda$ is a model for $\Psi$ and for $\neg \Phi $.

There is an alternative way to look at $\Ione$, as a higher-order $\Ithree$-embedding. 
A second-order language is formally a two-sorted language, where one sort is interpreted by elements of the model, and the other sort is interpreted by subsets of the model (usually the first sort is lowercase and the second is uppercase). 
Both variables and parameters have two sorts. 
So, for example, if $\varphi$ has no quantifiers, $V_\lambda\vDash\forall X\ \forall x\ \varphi(x,X,a,A)$ means ``$\forall X\subseteq V_\lambda\ \forall x\in V_\lambda\ \varphi(x,X,a,A)$'', where $a\in V_\lambda$ and $A\subseteq V_\lambda$. 
If a second-order formula does not have quantifiers with uppercase variables, then it is $\Delta^1_0$ (or $\Sigma^1_0$ or $\Pi^1_0$). 
In a similar way to first-order formul\ae, a formula is $\Sigma^1_n$ if there are $n$ alternations of $\exists$ and $\forall$ quantifiers with uppercase variable, the first one being a $\exists$. 

Now, as $V_{\lambda+1}$ is just $\mathscr{P}(V_{\lambda})$, if there is a $j \colon V_\lambda\prec V_\lambda$, for any $X\in V_{\lambda+1}$ we can define $j^+(X)=\bigcup_{\alpha<\lambda}j(X\cap V_\alpha)$, so that $j$ extends to $j^+ \colon V_{\lambda+1}\to V_{\lambda+1}$. 
With this in mind, it makes sense now to ask whether $j$ preserves second-order formul\ae, i.e., whether $V_\lambda\vDash\varphi(a,A)$ iff $V_\lambda\vDash\varphi(j(a), j^+(A))$. 
If $j$ preserves all second-order formul\ae then, clearly, $j^+ \colon V_{\lambda+1}\prec V_{\lambda+1}$. 
We say that $j$ is a $\Sigma^1_n$-elementary embedding if it preserves $\Sigma^1_n$ formul\ae. 
If $j$ is a $\Sigma^1_n$-elementary embedding for any $n\in\omega$, then $j^+$ witnesses $\Ione$. 
The other direction also holds: if $j\colon V_{\lambda+1}\prec V_{\lambda+1}$, then the extension of $j\upharpoonright V_\lambda$ is $j$ itself (see Lemma 3.4 in~\cite{Dimonte}), so every $\Ithree$-embedding, if it can be extended to a $\Ione$-embedding, it can be extended in a unique way.
It is a standard fact that if $j$ witnesses $\Ithree$, then it is a $\Sigma^1_0$-elementary embedding (this is a consequence of the more general Lemma~\ref{lem:application}). 
The general concept of $\Sigma^1_n$-elementary embeddings appears first in the paper by Laver~\cite{Laver}, building on Martin's seminal work.

In order to state and prove the results that follow, we introduce the following 

\begin{defin}\label{def:E(lambda)}
\( \mathscr{E} (  \lambda ) \) is the set of all \( j \colon V_  \lambda  \prec V_  \lambda \), and \( \mathscr{E}_n (  \lambda ) \) is the set of all \( j \in \mathscr{E} (  \lambda ) \) that are \( \Sigma ^1_n \)-elementary.
\end{defin}

Therefore \( \mathscr{E} (  \lambda )  = \mathscr{E}_0 (  \lambda )  \supseteq \mathscr{E}_1 (  \lambda ) \supseteq \mathscr{E}_2 (  \lambda ) \supseteq \dots \),  \( \Ithree (  \lambda  ) \) means \( \mathscr{E} (  \lambda ) \neq \emptyset \), and \( \Ione (  \lambda ) \) is \( \bigcap_n \mathscr{E} _n (  \lambda  ) \neq \emptyset \).

Now we can characterize $\omega$-superstrongness within this template:

\begin{teo}[Martin,~\cite{Martin}]\label{th:I2}
Any $j \in \mathscr{E}_1 (  \lambda )$  can be extended to an $i \colon V\prec M$ such that $V_\lambda\subseteq M$.
Conversely, if $i \colon V\prec M$ is such that $V_\lambda\subseteq M$, with $\lambda$ supremum of the critical sequence, then $i \upharpoonright V_\lambda \in \mathscr{E}_1 (  \lambda )$.
\end{teo}

In view of Theorem~\ref{th:I2}, we write $\Itwo(\lambda)$ for one of the two following equivalent statements:
	\begin{itemize}
	\item there is a $j \colon V\prec M$ such that $V_\lambda\subseteq M$, where $\lambda$ is the supremum of the critical sequence of $j$;
	\item $\mathscr{E}_1 (  \lambda ) \neq \emptyset$.
	\end{itemize}

Going back to \( \Sigma ^1_n \)-embeddings, there is  a hierarchy of hypotheses of length $\omega+1$, that starts with $\Ithree$, $\Itwo$, and then $\Sigma^1_2$-, $\Sigma^1_3$-, $\dots$ elementarity up to $\Ione$. 
Do they really form a proper hierarchy? 

\begin{teo}[Laver, Martin,~\cite{Laver},~\cite{Martin}]\label{LaverMartin}
For every \( n \in \omega  \)
\begin{enumerate}
\item 
\( \mathscr{E}_{2n + 1} (  \lambda ) = \mathscr{E}_{2n + 2} (  \lambda ) \);
\item 
if  $ j \in \mathscr{E}_{2n + 1} (  \lambda ) $ and \( \kappa = \crt ( j ) \), then \( C_n = \{ \lambda ' < \kappa : \mathscr{E}_{2n} (  \lambda ' ) \neq \emptyset \} \) is \( \omega \)-club in \( \kappa \).
	\end{enumerate}
\end{teo}
Thus the existence of a $\Sigma^1_{2n+1}$-embedding from $V_\lambda$ to itself strongly implies the existence of a $\Sigma^1_{2n}$-embedding from $V_\lambda$ to itself.

The following result is Corollary 5.24 in~\cite{Dimonte}. 
Since the proof in that paper follows from a long-winded argument, for the reader's convenience we present a self-contained one.

\begin{cor}\label{smallstep}
\( \Ione ( \lambda ) \) strongly implies \( \forall n \ \mathscr{E}_n ( \lambda ) \neq \emptyset \).
\end{cor}

\begin{proof}
If $j \colon V_{\lambda + 1} \prec V_{\lambda + 1}$ witnesses $\Ione$, then \( j \restriction V_  \lambda \) is in $\mathscr{E}_{2n + 1} (  \lambda ) $ for any $n\in\omega$.
Let $C_n\subseteq\crt(j)$ be as in Theorem~\ref{LaverMartin};  as $\crt(j)$ is regular, $\bigcap_{n\in\omega}C_n\neq\emptyset$, therefore there is a $\lambda'<\crt(j)<\lambda$ such that \( \mathscr{E}_{2n} (  \lambda ' ) \neq \emptyset \) for any $n\in\omega$.
\end{proof}

Figure~\ref{fig:firsthierarchy} summarizes the situation until now, where all vertical arrows are strong implications.
\begin{figure}
 \centering
\begin{tikzpicture} \label{firsthierarchy}
 \node (I1){$\Ione ( \lambda ) \leftrightarrow \bigcap_n \mathscr{E} _n (  \lambda  ) \neq \emptyset $};
 \node (almostI1)[below of = I1]{$  \forall n\  \mathscr{E} _n (  \lambda  ) \neq \emptyset $}
	edge[<-](I1);
 \node (dots)[below of = almostI1]{$\vdots$}
 edge[<-](almostI1);
 \node (34)[below of = dots]{$ \mathscr{E} _4 (  \lambda  ) =  \mathscr{E} _3 (  \lambda  ) \neq \emptyset $}
 edge[<-](dots);
 \node (12)[below of = 34]{$\Itwo ( \lambda ) \leftrightarrow \mathscr{E} _2 (  \lambda  ) =  \mathscr{E} _1 (  \lambda  ) \neq \emptyset $}
 edge[<-](34);
 \node (base)[below of = 12]{$\Ithree ( \lambda ) \leftrightarrow  \mathscr{E} ( \lambda ) = \mathscr{E} _0 ( \lambda ) \neq \emptyset $}
 edge[<-](12);
\end{tikzpicture}
 \caption{}
 \label{fig:firsthierarchy}
\end{figure}

\section{Iterations of $\Ithree$}\label{sec:iterate}

If \( E \) is an extender in a transitive model \( M \) of \( \ZFC \), the iteration of length \( \nu \) is a commutative system of transitive models, extenders, and elementary embeddings \( \langle ( M_ \alpha , E_ \alpha , j_{ \beta , \alpha } ) : \beta \leq \alpha \in \nu \rangle \) defined as follows:
\begin{itemize}
\item
\( M_0 = M \), \( E_0 = E \), and \( j_{ 0 , 0} \) is the identity on \( M \),
\item
\( j_{ \alpha , \alpha + 1} \colon M_ \alpha \prec \Ult ( M_ \alpha , E_ \alpha )^{M_ \alpha } = M_{ \alpha + 1 } \) is the ultrapower embedding, and \( \Ult ( M_ \alpha , E_ \alpha )^{M_ \alpha } \) is the ultrapower of \( M_ \alpha \) via \( E_ \alpha \) computed in \( M_ \alpha \),
\item
\( E_{ \alpha + 1 } = j_{ \alpha , \alpha + 1} ( E_ \alpha ) \), and for \( \beta \leq \alpha \) we set \( j_{ \beta , \alpha + 1} = j_{ \alpha , \alpha + 1} \circ j_{ \beta , \alpha } \), 
\item
for \( \gamma \) limit \( M_ \gamma \) is the direct limit of the \( M_ \alpha \)s for \( \alpha < \gamma \).
\end{itemize}
If \( \nu \subseteq M \) one verifies by induction on \( \alpha \in \nu \)  that \( M_ \alpha \) is well-founded---see the proof of Lemma 19.5 in~\cite{Jech}.
In particular, if \( M \) is a proper class then \( \nu \) can be replaced by \( \Ord \).

In 1978 Martin showed that the determinacy of $\boldsymbol{\Pi}^1_2$ sets followed from the existence of an elementary embedding $j \colon V_\lambda\prec V_\lambda$ that is iterable~\cite{Martin}.
Before considering iterability, we need to define this concept rigorously. 
The crux of the matter is that $j$ is a proper class of $V_\lambda$, and it cannot be a definable class by a generalization of Kunen's Theorem by Suzuki~\cite{Suzuki}. 
So the hypothetical $j_{1,2}$ cannot be calculated directly as a $j(j)$, since $j$ is not an element of $V_\lambda$, but it cannot even be calculated indirectly, as \( j \) is not an ultrapower embedding via some extender.
The idea is then to exploit the fact that $\lambda$ has cofinality $\omega$, defining the first iterate as $j^+(j)$, and at limit stages finding a way to define $j_{0,\omega}^+(j)$. 
In this process, however, we do not have the assurance that we can always prolong the iterate, and in fact we will see in Theorem~\ref{th:step} that there are cases where it cannot even reach the $\omega$-limit. 

It is worthwhile noticing that climbing the hierarchy of rank-into-rank axioms up to $\Izero$ and beyond, this problem disappears again.
In fact:  \( \Itwo (  \lambda ) \) is witnessed by $\omega$-superstrong embeddings defined by an extender, which is iterable by the argument at the beginning of this section; \( \Ione (  \lambda ) \) is taken care by Proposition~\ref{prop:Ioneisiterable} below; \( \Izero (  \lambda ) \) can be witnessed by \emph{proper embeddings} (see \cite{Woodin,Dimonte2,Dimonte3}) which are in fact iterable, as they are defined from a normal ultrafilter. 
It seems therefore that iterability is a problem peculiar to $\Ithree$.

\begin{prop}\label{prop:Ioneisiterable}
 Every $j \colon V_{\lambda+1}\prec V_{\lambda+1}$ is iterable.
\end{prop}

\begin{proof}
 Let $j \colon V_{\lambda+1}\prec V_{\lambda+1}$. 
 In particular $j\upharpoonright V_\lambda=k$ is a $\Sigma^1_1$-embedding, therefore it can be extended to an embedding $i \colon V\prec M$ with $V_\lambda\subseteq M$. 
 Now, $i$ is iterable, so $i_\alpha=i_{\alpha,\alpha+1}$, the $\alpha$-th iterate, and $i_{0,\alpha}$, the limit embedding, are defined for any $\alpha$ ordinal, and $M_\alpha$, the $\alpha$-th model of the iteration, is well-founded. 
 We want to define an iterate for $j$. 
 Since $j(\lambda)=\lambda$, we have that $j(k) \colon V_\lambda\prec V_\lambda$, so we can define $j_2=j(k)^+$, i.e., the extension of $j(k)$ to $V_{\lambda+1}$, $j_3=j_2(j_2\upharpoonright V_\lambda)^+$ and so on.
 Now, as the extension of an $\Ithree$-embedding is uniquely defined, $i_n\upharpoonright V_{\lambda+1}=j_n$, and the $\omega$-th model of the iteration of $j$ is $(V_{i_{0,\omega}(\lambda)+1})^{M_\omega}$ (see proof of Proposition~\ref{prop:extender}). 
 We can define therefore $j_\omega=j_{0,\omega}(k)^+$, with $j_{0,\omega}(k) \colon (V_{i_{0,\omega}(\lambda)})^{M_\omega}\prec (V_{i_{0,\omega}(\lambda)})^{M_\omega}$ and $i_\omega\upharpoonright (V_{i_{0,\omega}(\lambda)+1})^{M_\omega}=j_\omega$. 
Finally, ``$k$ is $\Sigma^1_n$'' is a $\Sigma^1_{n+2}$-formula, (see Lemma 2.1 in \cite{Laver}) and therefore by elementarity of the $j_{0,\omega}$, $j_{0,\omega}(k)$ is a $\Sigma^1_n$-embedding for any $n\in\omega$, and therefore an $\Ione$-embedding. So we can continue indefinitely the construction, and $j$ is iterable.
\end{proof}

The following lemmas provide the key construction for iterates of rank-into-rank embeddings:

\begin{lem}\label{lem:cofinalandamenable=>Sigma1}
Let $M,N\vDash \ZFC$ be transitive sets, $\pi \colon M\prec N$, and  $X\subseteq M$. 
If
\begin{itemize}
\item 
$\pi$ is cofinal, i.e., $\forall\beta\in N\cap \Ord$ $\exists\alpha\in M\cap \Ord$ $\pi(\alpha)>\beta$;
\item 
$X$ is amenable for $M$, i.e., $\forall\alpha\in M\cap \Ord$ $X\cap (V_\alpha)^M\in M$,
\end{itemize}
define $\pi^+ ( X ) = \bigcup_{\alpha \in M \cap \Ord} \pi ( X \cap ( V_\alpha ) ^M) $. 
Then \( \pi^+ ( X )  \) is amenable for \( N \) and $\pi \colon ( M , X ) \prec_1 ( N , \pi^+ (X) ) $.
\end{lem}

\begin{proof}
Well known. 
A simple induction proves that $\pi \colon ( M , X ) \prec_0 ( N , \pi^+ ( X )) $, and every cofinal $\Sigma_0$-embedding is a $\Sigma_1$-embedding.
\end{proof}

The next result shows that with an assumption on the cofinality of $M\cap\Ord$ we can have full elementarity:

\begin{lem} \label{lem:application}
Let $M,N\vDash \ZFC$ be transitive sets.
Suppose $\pi \colon M\prec N$ is cofinal, that  $ \gamma = \cof ( M\cap \Ord ) \in M $, and that \( M^{< \gamma } \subseteq M \). 
Then $\pi \colon (M,X)\prec (N,\pi^+(X))$ for any $X\subseteq M$ amenable in $M$.
\end{lem}

\begin{proof}
We prove by induction on \( n \) that $\pi \colon ( M , X ) \prec_n ( N , \pi^+ ( X ) ) $ for any $ X \subseteq M$ amenable in $M$. 
The case \( n = 1 \) holds by Lemma~\ref{lem:cofinalandamenable=>Sigma1}, so we may assume that  the result holds for some \( n \geq 1 \) towards proving the result for \( n + 1 \).

Fix $F=\langle\kappa_\alpha : \alpha < \gamma \rangle$ cofinal in $M\cap \Ord$. 
As \( M^{< \gamma } \subseteq M \) then \( F \) is amenable for \( M \).
Let \( \Fml_n \) be the set of (codes for) \( \Sigma _n \)-formulas in the language of set theory augmented with a \( 1 \)-ary predicate \( \mathring{X} \).
Let
\begin{equation*}
 B=\{( \exists y \psi ( x , y ) ,\alpha ) \in \Fml_n \times \gamma : \psi \text{ is }\Pi_{n-1}\wedge ( M , X ) \vDash\exists y\in V_{\kappa_\alpha}\psi(x,y) \}. 
\end{equation*}
Note that \( B \) is amenable for \( M \).
If $\varphi(x) \in \Fml_n $ is  $\exists y\psi(x,y)$, then
\begin{multline*}
\forall x \bigl ( \varphi(x) \rightarrow \exists \alpha<\gamma \, ( \varphi ( x ) , \alpha ) \in \mathring{B} \bigr )
\\
{}\wedge \forall x \forall \kappa \forall\alpha<\gamma \bigl ( ( \varphi(x) , \alpha)\in \mathring{B} \wedge  ( \alpha , \kappa ) \in \mathring{F} \rightarrow\exists y\in V_{ \kappa } \, \psi(x,y) \bigr )
\end{multline*} 
is a \( \Pi_n \)-formula \( \mathsf{\Psi}_{ \varphi ( x ) } \) in the language of set theory augmented with predicates \( \mathring{X}, \mathring{F} , \mathring{B} \) that holds true in \( ( M , X , F , B ) \).
(The assumption that the cofinality of \( M \cap \Ord \) is singular is used to bound the quantifier \( \exists  \alpha < \gamma  \) so that \( \mathsf{\Psi}_{ \varphi ( x ) } \) is indeed a \( \Pi_1 \) formula when \( n = 1 \).)
The formulas \( \mathsf{\Psi}_{ \varphi ( x ) } \) with \( \varphi(x) \in \Fml_n \) describe that $B$ is exactly as defined, therefore we can say that the $\Pi_n$-theory of $(M,X,F,B)$ ``knows'' the definition of $B$. 
Since 
\begin{equation*}
 \pi \colon (M,X,F,B)\prec_n (N,\pi^+(X),\pi^+(F),\pi^+(B)), 
\end{equation*}
then $\pi^+(B)$ is as expected, i.e., 
\begin{multline*}
 \pi^+ ( B ) = \{ ( \exists y \psi(x,y) , \alpha ) \in \Fml_n \times \pi ( \gamma ) : \psi\text{ is }\Pi_{n-1}
 \\
 {} \wedge (N,\pi^+(X))\vDash\exists y\in V_{\pi( F ( \alpha ) )}\psi(x,y)\}, 
\end{multline*}
so $(N,\pi^+(X),\pi^+(B))\vDash\exists \alpha< \pi ( \gamma ) \ ( \varphi(x) , \alpha ) \in \mathring{B}  $ iff $(N,\pi^+(X))\vDash\varphi(x)$. 

We are now ready to show that $\pi$ preserves all \( \Pi_{n + 1} \) formulas and hence it is $\Sigma_{n+1}$-elementary.
If $\varphi$ is a $\Sigma_n$ formula then 
\[
\begin{split}
(M,X)\vDash\forall x \varphi ( x ) & \leftrightarrow ( M , X , B ) \vDash \forall x \exists \alpha<\gamma \, ( \varphi(x) , \alpha ) \in \mathring{B}
\\
 & \leftrightarrow ( N , \pi^+ ( X ) , \pi^+ ( B ) ) \vDash \forall x \exists \alpha < \pi ( \gamma) \, ( \varphi ( x ) ,\alpha ) \in \mathring{B}
 \\
 & \leftrightarrow ( N , \pi^+ ( X ) ) \vDash\forall  x \varphi ( x )
\end{split}
\]
where the second equivalence follows from \( \pi \colon ( M , X , B ) \prec_1 ( N , \pi^+ ( X ) , \pi^+ ( B ) ) \) by Lemma~\ref{lem:cofinalandamenable=>Sigma1} and therefore it preserves \( \Pi_1 \) formulas.
\end{proof}

When $M=N=V_\lambda$ then \( \cof ( M \cap \Ord ) = \omega \) so that the hypothesis \( M^{< \gamma  } \subseteq M \) in the statement of Lemma~\ref{lem:application} holds automatically.
Lemma~\ref{lem:application} for $M=N=V_\lambda$ appears in several places without proof (e.g.,~\cite{Kanamori,Laver}), but only in~\cite{Laver2} there is a proof of that. 
Unfortunately, as it is written in~\cite{Laver2} there is a small gap: the proof is based on defining $j^+$ first on Skolem functions, but it is not clear why $j^+(f)$ should be total for any $f$ Skolem function. 
This problem is solved as Claim 3.7 in~\cite{Dimonte}. 
The proof above is instead an argument by Woodin found on MathOverflow~\cite{MOF}. 

Lemma~\ref{lem:cofinalandamenable=>Sigma1} shows how to extend an elementary embedding to amenable subsets.
A simple calculation proves that such extensions behave as expected between each other:

\begin{lem} \label{lem:operations}
Let $M_1 , M_2 ,N_1,N_2$ be transitive sets and models of \( \ZFC \).
\begin{enumerate} 
\item
If $j \colon M_1 \prec N_1$ and  $\pi \colon N_1\prec N_2$ are cofinal, and $X\subseteq M_1 $ is amenable for $M_1$, then $\pi^+(j^+(X))=(\pi^+(j))^+(\pi^+(X))$.
\item
If $\pi_1 \colon M_1\prec N_1$, $\pi_2 \colon M_2\prec N_2$, $j_1 \colon M_1\prec M_2$, and $j_2 \colon N_1\prec N_2$ are cofinal and $\pi_2\circ j_1=j_2\circ \pi_1$, then $\pi^+_2\circ j^+_1=j^+_2\circ \pi^+_1$ on the sets amenable for $M_1$.
\[
\begin{tikzcd}
N_1 \ar[r, "j_2"] & N_2
\\
M_1 \ar[r, "j_1"] \ar[u, "\pi_1"] & M_2 \ar[u, "\pi_2"] 
\end{tikzcd}
\]
\end{enumerate}
\end{lem}

The next result is folklore.

\begin{lem}
If  $ j , k \colon V_\lambda\prec V_\lambda$, then $ j^+ ( k ) \colon V_\lambda \prec V_\lambda$. 
\end{lem}

\begin{proof}
Let \( \Fml \) be the set of all (codes of) first order formulas in the language of set theory. 
Note that  ``\( k \) is an elementary embedding from \( V_  \lambda  \) to itself'' amounts to say that \( V_  \lambda  \vDash \mathsf{\Upsilon}_{ \varphi ( x ) }\) for every \( \varphi ( x ) \in \Fml \), where \( \mathsf{\Upsilon}_{ \varphi ( x ) } \) is 
\[
 \forall x ( \varphi ( x ) \rightarrow \exists  y ( ( x , y ) \in \mathring{k} \wedge \varphi ( y ) )
\]
with \( \mathring{k} \) a binary predicate predicate for \( k \).  
Since \( j \colon ( V_  \lambda , k ) \to ( V_  \lambda , j^+(k) ) \) by Lemma~\ref{lem:application}, it follows that  $j^+(k) \colon V_\lambda \prec V_\lambda$. 
\end{proof}

In particular $j^+(j)=j_1$ will be an embedding with critical point $\crt(j^+(j))=j(\crt(j))=\kappa_1$. 
Letting \( j_0 = j \) and $j_{n+1}=j^+(j_n)$ one proves by induction on \( n \)  that $ j_{n+1} = j_n^+(j_n)$. 
Note that $\crt(j_{n})=\kappa_n$. 
By induction on \( n \) it follows that $j_n ( \kappa_m ) = \kappa_{m+1}$ when $m\geq n$: 
\begin{equation*}
 j_n ( \kappa_m ) = j_{n-1} ( j_{n-1} ) ( j_{n-1} ( \kappa_{m-1} ) ) = j_{n-1} ( j_{n-1} ( \kappa_{m-1} ) ) = j_{n-1} ( \kappa_m ) = \kappa_{m+1} .
\end{equation*}
Let $M_\omega$ be the direct limit of the system $ \langle ( V_\lambda , j_{n,m} ) : n,m \in \omega , n<m \rangle $, where $j_{n,n+1}=j_n$ and $j_{n,m} = j_m \circ j_{m-1}\circ\dots\circ j_n$. 
If $M_\omega$ is well-founded, then it is defined $j_{n,\omega} \colon V_\lambda\prec M_\omega$ for any $n\in\omega$. Now, $j_{0,\omega}$ is cofinal: Let $\alpha\in M_\omega$. 
Then there exist $n\in\omega$ and $\beta\in\lambda$ such that $\alpha=j_{n,\omega}(\beta)$. 
Let $m$ be such that $\beta<\kappa_m$ and $m>n$. 
Then $j_{0,n}(\kappa_{m-n})=\kappa_m>\beta$, and 
\begin{equation*}
 j_{0,\omega}(\kappa_{m-n})=j_{n,\omega}(j_{0,n}(\kappa_{m-n}))>j_{n,\omega}(\beta)=\alpha . 
\end{equation*}
Therefore we can define $j_\omega=j_{0,\omega}^+(j) \colon M_\omega\prec M_\omega$, and then $j_{\omega+1}=(j_\omega)^+(j_\omega) \colon M_\omega\prec M_\omega$, and so on. 
At each limit point we ask whether the direct limit is well-founded, and if so we continue, otherwise we stop.

Note that, differently than in the case $j \colon V\prec M$, the model $M_\alpha$ is the same as $M_{\alpha+1}$, so they are either both well-founded or not. 
In other words, the construction can stop only at limit ordinals. 
We say that $j$ is $\alpha$-iterable, then, if the construction does not stop at the $\omega\cdot\alpha$-th step, i.e, if $M_{\omega\cdot\alpha}$ is well-founded, and ${<}\alpha$-iterable if $M_{\omega\cdot\beta}$ is well-founded for any $\beta<\alpha$. 
As usual, we identify $M_\beta$ with its transitive collapse, when well-founded. 
We write $\Ithree_\alpha$ to indicate the existence of an $\alpha$-iterable embedding from $V_\lambda$ to itself, and $\Ithree_{<\alpha}$ for the existence of a ${<}\alpha$-iterable embedding. 
We say that $j$ is iterable, and we indicate the relative hypothesis with $\Ithree_\infty$, if it is $\alpha$-iterable for any $\alpha$ ordinal.

If $j$ is 1-iterable, as $j_m(\kappa_n)=\kappa_n$ for any $m>n$, we have that $\crt(j_{n,\omega})=\kappa_n$, so if $x\in V_{\kappa_n}$, $j_{n+1,\omega}(x)=x\in M_\omega$. 
This means that $j_{0,\omega}(\kappa_0)=\lambda$ and $V_\lambda\subseteq M_\omega$, so $M_\omega$ is actually ``taller'' then $V_\lambda$.
As $(V_\lambda)^{M_\omega}=V_\lambda$, and this implies that $V_\lambda\in M_\omega$. 
In the same way, if $j$ is 2-iterable then $M_\omega\in M_{\omega\cdot 2}$, and, more generally,  if $j$ is $\beta$-iterable then $M_{\omega\cdot\alpha}\in M_{\omega\cdot\beta}$ for any $\alpha<\beta$.

Suppose \( E \) is an extender in a transitive model \( M \) and let \( M_ \alpha  \) denote the \( \alpha \)th model of the  iteration.
It is a standard result in inner model theory that if \( M_ \alpha  \) is well-founded for every \( \alpha < \omega _1 \), then every \( M_ \alpha  \) is well-founded.
This holds also in our situation.

\begin{prop}\label{prop:iterability}
For $j \colon V_\lambda\to V_\lambda$, the following are equivalent:
\begin{itemize}
\item 
$\Ithree_\infty(j)$, that is $M_\alpha$ is well-founded for any $\alpha$ ordinal, i.e., $j$ is iterable,
\item 
$\Ithree_{\omega_1}(j)$, that is $M_{\omega_1}$ is well-founded, i.e., $j$ is $\omega_1$-iterable,
\item 
$\forall\beta<\omega_1\ \Ithree_\beta(j)$, that is $M_\beta$ is well-founded for any $\beta<\omega_1$, i.e., $j$ is ${<}\omega_1$-iterable.
 \end{itemize}
\end{prop}

\begin{proof}
 Only one direction is not obvious. 
 So suppose that $\forall\beta<\omega_1\ \Ithree_\beta(j)$ and that there exists $\theta$ such that $M_\theta$ is ill-founded. 
 Without loss of generality, we can assume that $\theta$ is least, limit and $\geq\omega_1$. 
 We will prove that this leads to a contradiction.

 Pick $\alpha$ large enough so that $\langle M_\nu \colon \nu<\theta\rangle\in V_\alpha$ together with witnesses of the ill-foundedness of $M_\theta$. 
 Let $\pi \colon \mathcal{P}\to V_\alpha$ the inverse of the collapse such that $j$, $\langle M_\nu : \nu<\theta\rangle$, $V_\lambda$, $\theta$, $\kappa_n$ for all $n\in\omega$ and the witnesses of the ill-foundedness of $M_\theta$ are all in the range of $\pi$, with $\mathcal{P}$ countable. 
 Let $\bar{M}_0=\pi^{-1}(V_\lambda)$, $\bar{\jmath}_0=\pi^{-1}(j)$ and $\bar{\theta}=\pi^{-1}(\theta)$, and let $\langle \bar{M}_\nu : \nu\leq\bar{\theta}\rangle$ be the iteration of $(\bar{M}_0,\bar{\jmath}_0)$. 
 Then $\bar{M}_{\bar{\theta}}$ is ill-founded in $\mathcal{P}$.

We want all the models of the iterates of $j$ and $\bar{\jmath}$ to satisfy the hypothesis of Lemma~\ref{lem:application}, so to have singular height. 
But note that, as $\langle\kappa_n : n\in\omega\rangle$ is cofinal in $V_\lambda$, for any $\nu<\theta$ we have that $\langle j_{0,\nu}(\kappa_n) : n\in\omega\rangle$ is cofinal in $M_\nu$, as $j_{0,\nu}$ is a $\Sigma^1_0$-elementary embedding by Lemma~\ref{lem:application} and $\langle j_{0,\nu}(\kappa_n) : n\in\omega\rangle=j_{0,\nu}^+(\langle\kappa_n : n\in\omega\rangle)$. 
Let $\mu_n=\pi^{-1}(\kappa_n)$. 
Then also $\langle\mu_n : n\in\omega\rangle$ is cofinal in $\bar{M}_0\cap \Ord$ and for every $\nu<\bar{\theta}$ $\langle \bar{\jmath}_{0,\nu}(\mu_n) : n\in\omega\rangle$ is cofinal in $\bar{M}_\nu\cap \Ord$. 
So $\bar{M}_\nu\cap \Ord$ has cofinality $\omega$, for every $\nu\leq\bar{\theta}$. 
As $\bar{\theta}$ is countable, all the $M_\alpha$ and $\bar{M}_\alpha$ are well-founded for $\beta<\bar{\theta}$ for case assumption. 

 We build by induction $\pi_\nu$, for every $\nu\leq\bar{\theta}$, such that:
\begin{enumerate}
\item 
$\pi_\nu \colon \bar{M}_\nu\prec M_\nu$ and it is cofinal;
\item 
$\pi_\nu\circ\bar{\jmath}_{\delta,\nu}=j_{\delta,\nu}\circ\pi_\delta$ for every $\delta<\nu$;
\item 
$\pi_\nu^+(\bar{\jmath}_\nu)=j_\nu$.
\end{enumerate}
\[
\begin{tikzpicture}
 \node (barM0) {$\bar{M}_0$};
 \node (barM02)[right of = barM0] {$\bar{M}_0$}
 edge [<-] node[below]{$\bar{\jmath}_0$}(barM0);
 \node (dots)[right of = barM02] {$\dots$};
 \node (barMalpha)[right of = dots] {$\bar{M}_\delta$};
 \node (dots2)[right of = barMalpha] {$\dots$};
 \node (barMnu)[right of = dots2] {$\bar{M}_\nu$}
 edge [<-,out=-135,in=-45] node[below]{$\bar{\jmath}_{\delta,\nu}$}(barMalpha);
 \node (barMnu2)[right of = barMnu, xshift=3mm] {$\bar{M}_\nu$}
 edge [<-] node[below]{$\bar{\jmath}_{\nu}$}(barMnu);
 \node (dots3)[right of = barMnu2] {$\dots$};
 \node (barMtheta)[right of = dots3] {$\bar{M}_{\bar{\theta}}$};
 \node (M0)[above of = barM0] {$V_\lambda$}
 edge [<-] node[left]{$\pi_0$}(barM0);
 \node (M02)[above of = barM02] {$V_\lambda$}
 edge [<-] node[above]{$j$}(M0)
 edge [<-] node[left]{$\pi_1$}(barM02);
 \node (dots3)[above of = dots] {$\dots$};
 \node (Malpha)[above of = barMalpha] {$M_\delta$}
 edge [<-] node[left]{$\pi_\delta$}(barMalpha);
 \node (dots4)[above of = dots2] {$\dots$};
 \node (Mnu)[above of = barMnu] {$M_\nu$}
 edge [<-] node[left]{$\pi_\nu$}(barMnu)
	 edge [<-,out=135,in=45] node[above]{$j_{\delta,\nu}$}(Malpha);
 \node (Mnu2)[above of = barMnu2] {$M_\nu$}
 edge [<-] node[left]{$\pi_{\nu+1}$}(barMnu2)
	 edge [<-] node[above]{$j_{\nu}$}(Mnu);
 \node (dots5)[above of = dots2] {$\dots$};
 \node (Mbartheta)[above of = barMtheta] {$M_{\bar{\theta}}$}
 edge [<-] node[left]{$\pi_{\bar{\theta}}$}(barMtheta);
 \node (dots6)[right of = Mbartheta] {$\dots$};
 \node (Mtheta)[right of = dots6] {$M_\theta$}
 edge [<-] node[below]{$\pi$}(barMtheta);
\end{tikzpicture} 
\]
If this can be achieved, we easily reach a contradiction: In $M_{\theta}$ there is a sequence that witnesses that $M_\theta$ is ill-founded, and by construction of $\pi$ such witnesses are in the range of $\pi$, therefore also $\bar{M}_{\bar{\theta}}$ is ill-founded. 
But then, by elementarity via $\pi_{\bar{\theta}}$, also $M_{\bar{\theta}}$ is ill-founded, a contradiction since $\bar{\theta}$ is countable and we assumed that all the $M_\alpha$ with $\alpha$ countable are well-founded.

For $\nu=0$, let $\pi_0=\pi\upharpoonright\bar{M}_0$. 
Then, of course, $\pi_0 \colon \bar{M}_0\prec V_\lambda$. 
It is cofinal because for all $n$, $\kappa_n\in\ran(\pi_0)$. 
Note that by elementarity $\bar{M}_0=(V_\eta)^{\mathcal{P}}$ for some $\eta$; so if $X\subseteq M_0$ and $X \in \mathcal{P} $, then $X$ is amenable in $\bar{M}_0$, and $\pi_0^+(X)=\pi(X)$, therefore $\pi_0^+(\bar{\jmath}_0)=\pi(\bar{\jmath}_0)=j$.

Let $\nu$ be a limit ordinal. 
Let $x\in\bar{M}_\nu$. 
Then there exist $\delta<\nu$ and $y\in\bar{M}_\delta$ such that $x=\bar{\jmath}_{\delta,\nu}(y)$. 
We define then $\pi_\nu(x)=j_{\delta,\nu}(\pi_\delta(y))$. 
It is easy to see that it is elementary and well defined. 
For any $n\in\omega$, then $\pi_\nu(\bar{\jmath}_{0,\nu}(\mu_n))=j_{0,\nu}(\pi_0(\mu_n))=j_{0,\nu}(\kappa_n)$, therefore $\pi_\nu$ is cofinal. 
Let $\delta<\nu$ and let $x\in\bar{M}_\delta$. 
Then $\pi_\nu(\bar{\jmath}_{\delta,\nu}(x))=j_{\delta,\nu}(\pi_\delta(x))$ by definition of $\pi_\nu$, therefore (2) holds. 
Also, $\pi_\nu^+(\bar{\jmath}_\nu)=\pi_\nu^+(\bar{\jmath}_{0,\nu}^+(\bar{\jmath}_0))=j_{0,\nu}^+(\pi_0^+(\bar{\jmath}_0))=j_{0,\nu}^+(j_0)=j_\nu$, the second equality holding by Lemma \ref{lem:operations}(2), so $\pi_\nu$ is as desired. 

Finally, let $\nu=\mu+1$. 
Then define $\pi_\nu=\pi_\mu$. 
(1) is immediate. 
For (2), we prove it for $\delta=\mu$, and the rest is by easy induction. 
Note that $\bar{\jmath}_{\mu,\nu}=\bar{\jmath}_\mu$ and $j_{\mu,\nu}=j_\mu$. 
Let $x\in\bar{M}_\nu$. 
Then 
\begin{equation*}
 \pi_\nu(\bar{\jmath}_{\mu,\nu}(x))=\pi_\nu(\bar{\jmath}_\mu(x))=\pi_\mu(\bar{\jmath}_\mu(x))=\pi_\mu^+(\bar{\jmath}_\mu)(\pi_\mu(x))=j_\mu(\pi_\mu(x))=j_{\mu,\nu}(\pi_\mu(x)).
\end{equation*}
For (3), $\pi_\nu^+(\bar{\jmath}_\nu)=\pi_\nu^+(\bar{\jmath}_\mu^+(\bar{\jmath}_\mu))=\pi_\mu^+(\bar{\jmath}_\mu)(\pi_\mu^+(\bar{\jmath}_\mu))=j_\mu(j_\mu)=j_\nu$, the second equality holding by Lemma \ref{lem:operations}(1).

But now there is $\pi_{\bar{\theta}} \colon \bar{M}_{\bar{\theta}}\prec M_{\bar{\theta}}$, with $M_{\bar{\theta}}$ well-founded because $\bar{\theta}$ is countable and $\bar{M}_{\bar{\theta}}$ ill-founded in $\mathcal{P}$, and therefore in $V$, contradiction.
\end{proof}

We say that  \( j \colon V_\lambda \prec V_\lambda \) is iterable if  \( j \) satisfies one of the three equivalent conditions of Proposition~\ref{prop:iterability}, and we denote with \( \langle ( M_ \alpha , j_ \alpha ) : \alpha \in \Ord \rangle \) the iteration of \( ( V_  \lambda  , j ) \).

\begin{lem}\label{lem:onestep}
Suppose \( N \) is a transitive model of \( \ZFC \) and that \( N \vDash E \) is a \( ( \kappa , \lambda ) \)-extender witnessing \( \kappa  \) is \( \omega \)-superstrong.
Let \( i_E \colon N \to \Ult ( N , E ) ^N = N' \) be the ultrapower embedding, let \( \kappa ' = i_E ( \kappa ) \), and let \( E' = i_E ( E ) \) so that \( E' \) is a \( ( \kappa ' , \lambda ) \)-extender witnessing in \( N' \) that \( \kappa ' \) is \( \omega \)-superstrong.
Then \( j ^+ ( j ) = i_{ E' } \upharpoonright V_  \lambda \), where \( j = i_E \upharpoonright V_  \lambda \).
\end{lem}

\begin{proof}
The result follows from the fact that a \( ( \kappa  ,  \lambda ) \)-extender \( F \) is completely determined by \( i_F \restriction V_  \lambda  \).
\end{proof}

The next result shows \( \Ithree_\infty \) sits between $\Ithree$ and $\Itwo$.

\begin{prop}\label{prop:extender}
$\Itwo(j)$ implies $\Ithree_\infty (j ) $,  for all $j \colon V_  \lambda  \to V_  \lambda $.
\end{prop}

\begin{proof}
Suppose $\Itwo(j)$ and let \( \kappa = \crt ( j ) \).
By Martin's Theorem~\ref{th:I2} let \( E \) be a $(\kappa,\lambda)$-extender such that $i \colon V\prec  \Ult ( V , E ) \supseteq V_\lambda$ and \( i \upharpoonright V_\lambda = j \).
Let \( \langle ( N_ \alpha , E_ \alpha , i_{ \beta , \alpha } ) : \beta \leq \alpha \in \Ord \rangle \) be the iteration of \( V \) via \( E \).
As argued at the beginning of this section, every \( N_ \alpha \) is well-founded.
It is enough to show that \( \langle ( M_ \alpha , j_ \alpha ) : \alpha \in \Ord \rangle \) is the iteration of \( ( V_  \lambda  , j ) \), where
\begin{align*}
 \lambda_ \alpha &=  i_{0 , \alpha } (  \lambda ) , & M_ \alpha & =  N_ \alpha \cap V_{ \lambda _ \alpha  } , & j_ \alpha & =  i_{ \alpha , \alpha + 1} \restriction M_ \alpha .
\end{align*}
This boils-down to show that \( j_ \alpha^+  ( j_ \alpha ) = j_ {\alpha + 1} \), which follows from Lemma~\ref{lem:onestep} and an easy induction on \( \alpha \).
\end{proof}

Therefore all the iterable embeddings are in consistency strength between $\Ithree$ and $\Itwo$. 
Are they strictly or strongly between them? 
The tools developed in~\cite{Laver} by Laver will be essential to prove that $\Ithree_\infty$ is strongly between $\Ithree$ and $\Itwo$:

\begin{prop}[Laver, Square root of elementary embeddings,~\cite{Laver}]\label{prop:laver}
Let $j \colon V_\lambda\prec V_\lambda$ and let $\kappa =\crt(j)$.
{\begin{enumerate}
\item  
If \( j \) is $\Sigma^1_1$-elementary (so $\Itwo(j)$) and  $\beta<\kappa$, then there exists $k \colon V_\lambda\prec V_\lambda$ such that $k^+(k)=j$ and $\beta<\crt(k)<\kappa$.
\item 
If \( j \) is  $\Sigma^1_{n+2}$, then for any $B\subseteq V_\lambda$ there exists $\lambda'< \kappa $ and $J \colon V_{\lambda'}\prec V_\lambda$ that is $\Sigma^1_n$ such that $B\in\ran(J^+)$.
 \end{enumerate} }
\end{prop}

Proposition~\ref{prop:laver}(1) is enough to prove that $\Itwo$ strictly implies $\Ithree_\infty$: Let $j \colon V_\lambda\prec V_\lambda$ be a $\Sigma^1_1$-elementary embedding with least critical point, so that $j$ witnesses $\Itwo$. 
Then by Proposition~\ref{prop:laver}(1) there is a $k \colon V_\lambda\prec V_\lambda$ such that $\crt(j)<\crt(k)$. 
Since $j$ was chosen with least critical point, $k$ cannot be $\Sigma^1_1$. 
But $k$ is iterable, as $k_1=j$ and therefore $k_{n+1}=j_n$ and their limit iterations are the same. 
But we can do better:

\begin{prop}\label{prop:I2>>iterable}
Let $j \colon V_\lambda\prec V_\lambda$ be a $\Sigma^1_1$-elementary embedding. 
Then there is a $\lambda'<\lambda$ and a $k \colon V_{\lambda'}\prec V_{\lambda'}$ that is iterable. 
In other words, $\Itwo$ strongly implies $\Ithree_\infty$.
\end{prop}

\begin{proof}
Let $j \colon V_\lambda\prec V_\lambda$ be $\Sigma^1_1$. 
Use Proposition~\ref{prop:laver}(2) above with $B=j$, and let $k$ with $\lambda'<\lambda$ and $J^+(k)=j$ (remember that being $\Sigma^1_1$ is the same as being $\Sigma^1_2$). 
Note that ``$j \colon V_\lambda\prec V_\lambda$'' is $\Delta_1^1$ in $V_\lambda$, so $k \colon V_{\lambda'}\prec V_{\lambda'}$. 
Let $\langle ( \bar{M}_\alpha,k_\alpha ) : \alpha<\gamma\rangle$ be an iteration of $k$ of length $\gamma<\omega_1$, and let $\langle ( M_\alpha , j_\alpha ) : \alpha<\omega_1\rangle$ be the iteration of length \( \omega _1 \) of $j$. 
Now the proof is the same as in Proposition~\ref{prop:iterability}: define for any $\nu\leq\gamma$ $J_\nu \colon \bar{M}_\nu\prec M_\nu$, cofinal, such that $J_\nu\circ k_{\alpha,\nu}=j_{\alpha,\nu}\circ J_\alpha$ for any $\alpha<\nu$ and $J_\nu^+(k_\nu)=j_\nu$. 
\[
\begin{tikzpicture}
 \node (barM0) {$V_{\lambda'}$};
 \node (barM02)[right of = barM0] {$V_{\lambda'}$}
 edge [<-] node[below]{$k$}(barM0);
 \node (dots)[right of = barM02] {$\dots$};
 \node (barMalpha)[right of = dots] {$\bar{M}_\alpha$};
 \node (dots2)[right of = barMalpha] {$\dots$};
 \node (barMnu)[right of = dots2] {$\bar{M}_\nu$}
 edge [<-,out=-135,in=-45] node[below]{$k_{\alpha,\nu}$}(barMalpha);
 \node (barMnu2)[right of = barMnu, xshift=3mm] {$\bar{M}_\nu$}
 edge [<-] node[below]{$k_{\nu}$}(barMnu);
 \node (dots3)[right of = barMnu2] {$\dots$};
 \node (barMtheta)[right of = dots3] {$\bar{M}_{\gamma}$};
 \node (M0)[above of = barM0] {$V_\lambda$}
 edge [<-] node[left]{$J_0$}(barM0);
 \node (M02)[above of = barM02] {$V_\lambda$}
 edge [<-] node[above]{$j$}(M0)
 edge [<-] node[left]{$J_1$}(barM02);
 \node (dots3)[above of = dots] {$\dots$};
 \node (Malpha)[above of = barMalpha] {$M_\alpha$}
 edge [<-] node[left]{$J_\alpha$}(barMalpha);
 \node (dots4)[above of = dots2] {$\dots$};
 \node (Mnu)[above of = barMnu] {$M_\nu$}
 edge [<-] node[left]{$J_\nu$}(barMnu)
	 edge [<-,out=135,in=45] node[above]{$j_{\alpha,\nu}$}(Malpha);
 \node (Mnu2)[above of = barMnu2] {$M_\nu$}
 edge [<-] node[left]{$J_{\nu+1}$}(barMnu2)
	 edge [<-] node[above]{$j_{\nu}$}(Mnu);
 \node (dots5)[above of = dots2] {$\dots$};
 \node (Mbartheta)[above of = barMtheta] {$M_{\gamma}$}
 edge [<-] node[left]{$J_\gamma$}(barMtheta);
\end{tikzpicture}
\]
If $\bar{M}_\gamma$ were ill-founded, then because of the elementarity of $J_\gamma \colon \bar{M}_\gamma\prec M_\gamma$ also $M_\gamma$ would be ill-founded, but $j$ is iterable, so $\bar{M}_\gamma$ is well-founded. 
This holds for any $\gamma<\omega_1$, and therefore by Proposition~\ref{prop:iterability} $k$ is iterable.
\end{proof}

Therefore the iterability hypotheses are not only between $\Ithree$ and $\Itwo$, but strongly below $\Itwo$. 
But this is where Laver's tools stop, as they are too coarse to actually be useful in investigating gaps under $\Ithree_\infty$. 
For this, we need tools that are partially borrowed from the ``classic'' iterability. 

\section{The iterability hierarchy} \label{hierarchy}
Recall from Definition~\ref{def:E(lambda)} that \( \mathscr{E} (  \lambda  ) \) is the set of all \( j \colon V_  \lambda \prec V_  \lambda \).
Then
\[
\mathscr{W}_\alpha ( \lambda ) = \{ j \in \mathscr{E} (  \lambda ) : M_{ \omega \cdot \alpha } \text{ is well-founded} \}
\]
is the set of all \( j \colon V_  \lambda \prec V_  \lambda \) that are \( \alpha \)-iterable.
Therefore
\[
\mathscr{E} ( \lambda ) = \mathscr{W}_0 ( \lambda ) \supseteq \mathscr{W}_1 ( \lambda ) \supseteq \dots \supseteq \mathscr{W}_\alpha ( \lambda ) \supseteq \dots .
\]
With this notations Propositions~\ref{prop:iterability} and~\ref{prop:I2>>iterable} become 
\[ 
\mathscr{E}_1 ( \lambda )  \subset \mathscr{W}_{ \omega _1}( \lambda ) = \bigcap_{ \alpha < \omega _1} \mathscr{W}_ \alpha ( \lambda ) = \mathscr{W}_ \beta ( \lambda ) 
\] 
for any \( \beta \geq \omega _1 \).
We will prove that there is a strong hierarchy below $\Ithree_\infty$:

\begin{teo}\label{th:step}
\begin{itemize}
\item 
If $\alpha<\omega_1$, then $\Ithree_{\alpha+1} (  \lambda ) $ strongly implies $\Ithree_\alpha  (  \lambda ) $, i.e., if $j \in \mathscr{W}_{\alpha+1} (  \lambda )$ then there is a $\lambda'<\lambda$ and an $e \in \mathscr{W}_{\alpha} ( \lambda' )$.
\item 
If $\nu<\omega_1$ is limit, then $\Ithree_\nu (  \lambda )$ strongly implies $\Ithree_{<\nu} (  \lambda )$, i.e., if $j \in \mathscr{W}_{\nu} (  \lambda )$, then there are a $\lambda' < \lambda$ and an $e \in \bigcap_{ \alpha < \nu} \mathscr{W}_{\alpha} ( \lambda' )$.
\item 
If $\nu\leq\omega_1$ is limit, then $\Ithree_{<\nu} (  \lambda )$ strongly implies $\forall\alpha<\nu\ \Ithree_\alpha (  \lambda )$, i.e., if $j \in \bigcap_{\alpha < \nu} \mathscr{W}_\alpha (\lambda)$, then there is a $\lambda'<\lambda$ such that $\mathscr{W}_\alpha (\lambda') \neq \emptyset$ for all \( \alpha<\nu \).
\end{itemize}
Moreover, for any instance of the above the $\lambda'$ that witnesses the strong implication can be cofinally high under $\lambda$, so for any $\eta<\lambda$ there exists $\eta<\lambda'<\lambda$ that witnesses the strong implication. 
We will call this cofinal strong implication.
\end{teo}

The hierarchy just above \( \Ithree \) will therefore look like this, where every vertical arrow is a cofinal strong implication:

\begin{figure}
 \centering
\begin{tikzpicture} \label{lasthierarchy}
 \node (dots){$\vdots$};
 \node (omegaplusone)[below of = dots]{$\mathscr{W} _{\omega+1} (  \lambda  ) \neq \emptyset $}
	edge[<-](dots);
 \node (omegaplus)[below of = omegaplusone]{$\mathscr{W} _\omega (  \lambda  ) \neq \emptyset $}
 edge[<-](omegaplusone);
 \node (omega)[below of = omegaplus]{$ \bigcap_n \mathscr{W} _n (  \lambda  ) \neq \emptyset $}
 edge[<-](omegaplus);
 \node (almostomega)[below of = omega]{$\forall n\in\omega\ \mathscr{W} _n (  \lambda  ) \neq \emptyset $}
 edge[<-](omega);
 \node (dots2)[below of = almostomega]{$\vdots$}
 edge[<-](almostomega);
  \node (w2)[below of = dots2]{$\mathscr{W} _2 (  \lambda  ) \neq \emptyset $}
 edge[<-](dots2);
   \node (w1)[below of = w2]{$\mathscr{W} _1 (  \lambda  ) \neq \emptyset $}
 edge[<-](w2);
  \node (w0)[below of = w1]{$\mathscr{W} _0 (  \lambda  ) = \mathscr{E} ( \lambda ) \neq \emptyset $}
 edge[<-](w1);
\end{tikzpicture}
 \caption{}
 \label{fig:lasthierarchy}
\end{figure}


\begin{proof}
We prove it gradually, as going-up the iterability hierarchy will introduce more and more problems. 
As usual, if $j \colon V_\lambda\prec V_\lambda$, then $\langle \kappa_n : n\in\omega\rangle$ is the critical sequence. 
When it exists, we call $\lambda_\alpha$ the height of $M_\alpha$, i.e., $M_\alpha\cap \Ord$. 
Therefore $\lambda_0=\lambda$, $\lambda_\omega=M_\omega\cap \Ord$ and so on.
At the same time, we define $\kappa_\alpha=j_{0,\alpha}(\kappa_0)$. 
These two sequences overlap often, as $\kappa_{\alpha+\omega}=\lambda_\alpha$ for any $\alpha$ limit, but there is a slight difference at limit of limit stages: the $\kappa_\alpha$ sequence is in fact continuous at limit points, so for example $\kappa_{\omega\cdot\omega}=\sup_{n\in\omega}\kappa_{\omega\cdot n}$, while the $\lambda_\alpha$ sequence is discontinuous, for example $\lambda_{\omega\cdot\omega}>\sup_{n\in\omega}\lambda_{\omega\cdot n}=\kappa_{\omega\cdot\omega}$, as $\kappa_{\omega\cdot\omega}\in M_{\omega\cdot\omega}$. 
In a certain sense, the $\kappa_\alpha$ sequence is finer than the $\lambda_\alpha$ sequence and continuously completes it.

\begin{claim}
$\Ithree_1$ cofinally strongly implies $\Ithree$, i.e., for any \( 1 \)-iterable $j \colon V_\lambda\prec V_\lambda$ there exists $\lambda'<\lambda$ and $k \colon V_{\lambda'}\prec V_{\lambda'}$, and for any $\eta<\lambda$ we can find such $\lambda'$ to be larger than $\eta$.
\end{claim}

\begin{proof}
Suppose there exists $j \colon V_\lambda\prec V_\lambda$ that is 1-iterable and let $M_\omega$ be the $\omega$-th iterated model, which is well-founded by assumption. 
Note that $j_{0,\omega}(\kappa_0)=\lambda$, therefore $\lambda$ is regular in $M_\omega$.

As $M_\omega$ is well-founded we build in $M_\omega$ a descriptive set-theoretic tree of approximations of an \( \Ithree \)-embedding. 
So let $T_1$ be defined as:
\begin{equation*}
 T_1 = \{ \langle \gamma_0 , ( e^0 , \gamma_1 ) , \dots , ( e^n , \gamma_{n+1} ) \rangle : \forall i< n \ e^i \subseteq e^{i+1} , \ \forall i \leq n \ e^i \colon V_{\gamma_i}\prec V_{\gamma_{i+1}} \}, 
\end{equation*}
as defined in $M_\omega$. 
Note that $T_1$ is a tree on $V_\lambda$ and $V_\lambda\in M_\omega$, so $T_1\in M_\omega$. 
Now, for any $n\in\omega$, $j\upharpoonright V_{\kappa_n}\in V_{\kappa_{n+1}}\subseteq V_\lambda\subseteq M_\omega$, therefore $\langle \kappa_0, (j\upharpoonright V_{\kappa_0},\kappa_1),\dots\rangle$ is a branch of $T_1$ in $V$. 
By absoluteness of well-foundedness there is therefore a branch of $T_1$ in $M_\omega$. 
Let $\langle ( e^n,\gamma_{n+1}) : n\in\omega \rangle$ together with $\gamma_0$ be a branch of $T_1$ in $M_\omega$. 
Let $\gamma_\omega=\sup_{n\in\omega}\gamma_n$ and $e=\bigcup_{n\in\omega}e^n$. 
Then $e \colon V_{\gamma_\omega}\prec V_{\gamma_\omega}$, and $\gamma_\omega<\lambda$, since $\lambda$ is regular in $M_\omega$. 
So $M_\omega\vDash \exists\lambda'<\lambda\ \exists e \colon V_{\lambda'}\prec V_{\lambda'}$, but then this is true also in $V$. 
This proved that 1-iterability strongly implies $\Ithree$.

Let now $\eta<\lambda$, and let $n\in\omega$ such that $\kappa_n>\eta$. 
Define a revised version of $T_1$, adding the condition that $\gamma_0>\eta$. 
Then the sequence $\langle \kappa_n, (j_n\upharpoonright V_{\kappa_n},\kappa_{n+1}),\dots\rangle$ is a branch of the revised $T_1$, so there is a branch also in $V$, and the $\gamma_\omega$ defined by the branch will be such that $\lambda'=\gamma_\omega>\gamma_0>\eta$.
\end{proof}

\begin{claim}
 $\Ithree_2$ cofinally strongly implies $\Ithree_1$, i.e., for any \( 2 \)-iterable $j \colon V_\lambda\prec V_\lambda$ there exists $\lambda'<\lambda$ and a \( 1 \)-iterable $k \colon V_{\lambda'}\prec V_{\lambda'}$, and for any $\eta<\lambda$ we can find such $\lambda'$ to be larger than $\eta$.
\end{claim}

\begin{proof}
Suppose now that $j \colon V_\lambda\prec V_\lambda$ is 2-iterable, so $M_{\omega+\omega}$ is well-founded. 
Again, we are going to define a tree $T_2$ whose branches are going to bring us $\Ithree$-embeddings. 
But we want more, since we want such embeddings to be 1-iterable. 
The solution is to build in $T_2$ at the same time a family of embeddings $\langle k_m : m\in\omega\rangle$ that commutes with the iterates of $e$, so that the $\omega$-limit of such family will witness that the $\omega$-limit of $e$ is going to be well-founded.
	
Let $\langle\gamma_0,(e^0,\gamma_1),\dots,(e^n,\gamma_{n+1}),\dots\rangle$ be a branch of $T_1$, in other words a sequence of approximations of an $\Ithree$-embedding $e$. 
We can define then approximations also for the iterates of $e$: for example 
\begin{equation*}
e_1\upharpoonright V_{\gamma_1}=e_1^1=e^1(e^0),\ e_1\upharpoonright V_{\gamma_2}=e^2_1=e^2(e^1),\ e_2\upharpoonright V_{\gamma_2}=e_2^2=e^2_1(e^1_1), 
\end{equation*}
where the subscript indicates the iteration number, the superscript the level up to which the iterate is approximate. 
Of course, $e^0_1$ and $e^1_2$ are the identity. 
In general, we define $e^{n+m+1}_{m+1}=e^{n+m+1}_m(e^{n+m}_m)$ for any $n,m\in\omega$, where $e^n_0=e^n$. 
So if we know $e$ up to $V_{\gamma_n}$, then we know any finite iterate up to $V_{\gamma_n}$.
	
Working in $M_{\omega+\omega}$ we define a tree $T_2$ on the set $M_\omega\in M_{\omega+\omega}$ in the following way:
\begin{enumerate}
\item the nodes of $T_2$ are of the form 
\begin{equation*}
\langle ( \gamma_0,\eta_0 ) , ( e^0,k_0^0,\gamma_1,\eta_1),(e^1,k_0^1,k_1^1,\gamma_2,\eta_2),\dots,(e^n,k_0^n,\dots,k_n^n,\gamma_{n+1},\eta_{n+1})\rangle;
\end{equation*}

\item $e^0\subseteq e^1\subseteq\dots\subseteq e^n$, i.e., $\forall i<n\ e^i\subseteq e^{i+1}$; 

\item $e^i \colon V_{\gamma_i}\prec V_{\gamma_{i+1}}$ for any $i\leq n$;

\item note that $k_m^l$ is defined only for $m\leq l\leq n$, and we want $k^m_m\subseteq k^{m+1}_m\subseteq\dots\subseteq k^n_m$ for any $m\leq n$, i.e. $\forall i<n,\ \forall l<n-i\ k^{i+l}_i\subseteq k^{i+l+1}_i$;

\item for any $m\leq l\leq n$, $k_m^l \colon V_{\gamma_l}\prec (V_{\eta_{l-m}})^{M_\omega}$;

\item for any $m\leq l<n$, $k^{l+1}_{m+1}\circ e^l_m=k^l_m$.
\end{enumerate}
	
Consider now
\begin{equation*}
\langle (\kappa_0,\kappa_\omega), (j\upharpoonright V_{\kappa_0},j_{0,\omega}\upharpoonright V_{\kappa_0},\kappa_1,\kappa_{\omega+1}), (j\upharpoonright V_{\kappa_1}, j_{0,\omega}\upharpoonright V_{\kappa_1}, j_{1,\omega}\upharpoonright V_{\kappa_1}, \kappa_2, \kappa_{\omega+2}), \dots \rangle.
\end{equation*}
We want to prove that this sequence is a branch of $T_2$. 
We need to prove that every element of the sequence is in $M_\omega$. 
Clearly $j\upharpoonright V_{\kappa_n}\in V_\lambda\in M_\omega$ for any $n\in\omega$. 
Moreover, also $j_{m,l}\upharpoonright V_{\kappa_n}\in V_\lambda$ for any $m\leq l\in\omega$ and $n\in\omega$. 
But then for any $m\leq n\in\omega$
\begin{equation*}
j_{m,\omega}\upharpoonright V_{\kappa_n}=j_{n+1,\omega}\circ j_{m,n+1}\upharpoonright V_{\kappa_n}=j_{n+1,\omega}(j_{m,n+1}\upharpoonright V_{\kappa_n})\in M_\omega
\end{equation*}
because $\crt(j_{n+1,\omega})=\kappa_{n+1}>\kappa_n$. 
Points (2), (3), (4) are immediate. 
For point (5), 
\begin{equation*}
j_{m,\omega}(\kappa_n)=j_{m,\omega}(j_{0,m}(\kappa_{n-m}))=j_{0,\omega}(\kappa_{n-m})=\kappa_{\omega+n-m}\quad\text{for any}\quad m\leq n\in\omega, 
\end{equation*}
so by elementarity $j_{m,\omega}\upharpoonright V_{\kappa_n} \colon V_{\kappa_n}\prec (V_{\kappa_{\omega+n-m}})^{M_\omega}$ for any $m\leq n\in\omega$. 
Finally, for point (6), notice that the iterate $j_m$ is $j_{m,m+1}$ for any $m\in\omega$, so if $x\in V_{\kappa_n}$ for some $n\leq m$, $n\in\omega$, $j_{m+1,\omega}(j_m(x))=j_{m,\omega}(x)$, and since $j_m\upharpoonright V_{\kappa_n} \colon V_{\kappa_n}\prec V_{\kappa_{n+1}}$, we have that 
\begin{equation*}
j_{m+1,\omega}\upharpoonright V_{\kappa_n+1}\circ j_m\upharpoonright V_{\kappa_n}=j_{m+1,\omega}\circ j_m\upharpoonright V_{\kappa_n}=j_{m,\omega}\upharpoonright V_{\kappa_n}.
\end{equation*}
So $T_2$ has a branch in $V$, and therefore it has a branch in $M_{\omega+\omega}$.
Consider such a branch, and let $\gamma_\omega=\sup_{n\in\omega}\gamma_n$, $\eta_\omega=\sup_{n\in\omega}\eta_n$, $e=\bigcup_{n\in\omega} e^n$ and $k_m=\bigcup_{n\in\omega}k^n_m$ for any $m\in\omega$. 
Then, as before, $e \colon V_{\gamma_\omega}\prec V_{\gamma_\omega}$ and $k_m \colon V_{\gamma_\omega}\prec (V_{\eta_\omega})^{M_\omega}$ for any $m\in\omega$. 
Note that $M_\omega=(V_{\lambda_\omega})^{M_{\omega+\omega}}$, therefore $\lambda=\kappa_\omega$ is regular also in $M_{\omega+\omega}$ and $\gamma_\omega<\lambda$. 
Also $\lambda_\omega=\kappa_{\omega+\omega}$ is regular in $M_{\omega+\omega}$, so $\eta_\omega<\lambda_\omega$. 
Consider now $N_\omega$ the $\omega$-iterated model of $e$. 
The picture is the following:
\[
\begin{tikzpicture}
 \node (M0) {$V_{\gamma_\omega}$};
 \node (M1)[right of = M0, xshift=3mm] {$V_{\gamma_\omega}$}
 edge [<-] node[below]{$e$}(M0);
 \node (dots)[right of = M1] {$\dots$}
 edge [<-] (M1);
 \node (Mn)[right of = dots, xshift=1cm] {$V_{\gamma_\omega}$}
 edge [<-] (dots);
 \node (Mn1)[right of = Mn, xshift=3mm] {$V_{\gamma_\omega}$}
 edge [<-] node[below]{$e_m$}(Mn);
 \node (dots2)[right of = Mn1] {$\dots$}
 edge [<-] (Mn1);
 \node (Nomega)[right of = dots2] {$N_\omega$};
 \node (upper)[above of = Nomega, yshift=1cm] {$(V_{\eta_\omega})^{M_\omega}$}
 edge [<-] node[right]{$k_\omega$}(Nomega)
	 edge [<-] (M0)
	 edge [<-] (M1)
	 edge [<-] node[left]{$k_m$}(Mn)
	 edge [<-] node[below right]{$k_{m+1}$}(Mn1);
\end{tikzpicture} 
\]	
By the properties of the direct limit, as the family of $k_m$ commutes with the iterates of $e$, there exists a $k_\omega \colon N_\omega\prec (V_{\eta_\omega})^{M_\omega}$. 
Then, by elementarity and well-foundedness of $M_\omega$, also $N_\omega$ is well-founded, and $e$ is a 1-iterable embedding below $\lambda$. 
Note that it is not necessary that $\gamma_\omega=\eta_0$, like in the branch in $V$.
	
Again, for any $\eta<\kappa_n<\lambda$ we can add to the definition of $T_2$ the condition $\gamma_0>\eta$, and the branch generated by $j_n$ instead of $j$ will make the proof work, so that we will find a 1-iterable $e \colon V_{\lambda'}\prec V_{\lambda'}$ with $\eta<\gamma_0<\gamma_\omega=\lambda'$.
\end{proof}
	
\begin{claim}
$\Ithree_3$ cofinally strongly implies $\Ithree_2$.
\end{claim}

\begin{proof}
This adds another layer of complexity. 
The tree $T_2$ as calculated in $M_{\omega\cdot 3}$ would give a 1-iterable embedding, but our aim is to build a 2-iterable embedding. 
The strategy of adding more witnesses to well-foundedness (so further $k_{\omega+n}$) cannot work in the same way, as $k_\omega$ is defined on $N_\omega$, and the initial segments of $N_\omega$ are known only when the whole $e$ is known, so it is not possible to build at the same time small approximations of $e$ and $k_\omega$. 
The solution is instead to build a 1-iterable embedding $h$ via $(T_2)^{M_\omega}$, so that $k_\omega^+(e_\omega)=h$, and an iteration argument will show that this is enough to prove that $e$ is 2-iterable.
	
Define $T_3$ in $M_{\omega\cdot 3}$ on $M_{\omega\cdot 2}\in M_{\omega\cdot 3}$ in the following way:
	\begin{enumerate}
	\item the nodes of $T_3$ are of the form 
	\begin{multline*}
	\langle (\gamma_0,\eta_0,\delta_0),(e^0,k_0^0,h^0,g^0_0,\gamma_1,\eta_1,\delta_1),(e^1,k_0^1,k_1^1,h^1,g_0^1,g_1^1,\gamma_2,\eta_2,\delta_2),\dots,\\
	(e^n,k_0^n,\dots,k_n^n,h^n,g_0^n,\dots,g_n^n,\gamma_{n+1},\eta_{n+1},\delta_{n+1})\rangle;
	\end{multline*}

	\item the sequence 
	\begin{equation*}
	\langle(\gamma_0,\eta_0),(e^0,k_0^0,\gamma_1,\eta_1),(e^1,k_0^1,k_1^1,\gamma_2,\eta_2),\dots,(e^n,k_0^n,\dots,k_n^n,\gamma_{n+1},\eta_{n+1})\rangle
	\end{equation*}
	is a node of $T_2$;

	\item the sequence 
	\begin{equation*}
	\langle(\eta_0,\delta_0),(h^0,g_0^0,\eta_1,\delta_1),(h^1,g_0^1,g_1^1,\eta_2,\delta_2),\dots,(h^n,g_0^n,\dots,g_n^n,\eta_{n+1},\delta_{n+1})\rangle
	\end{equation*}
	is a node of $(T_2)^{M_\omega}$, that is defined as $T_2$ but in $M_\omega$ and with all instances of $M_\omega$ replaced by $M_{\omega\cdot 2}$, so for example $h^n \colon (V_{\eta_n})^{M_\omega}\prec(V_{\eta_n})^{M_\omega}$ and $g^l_m \colon (V_{\eta_l})^{M_\omega}\prec(V_{\delta_{l-m}})^{M_{\omega\cdot 2}}$;

	\item for any $m\leq l<n$, $k^{l+1}_m(e^l_m)=h^{l-m}$.
	\end{enumerate}
As before, we can find a branch of $T_3$ in $V$ in the natural way, i.e., assign $\gamma_n=\kappa_n$, $\eta_n=\kappa_{\omega+n}$, $\delta_n=\kappa_{\omega+\omega+n}$, $e^n=j\upharpoonright V_{\kappa_n}$, $k^n_m=j_{m,\omega}\upharpoonright V_{\kappa_n}$, $h^n=j_\omega\upharpoonright (V_{\kappa_{\omega+n}})^{M_\omega}$ and $g^n_m=j_{\omega,\omega+m}\upharpoonright (V_{\kappa_{\omega+n}})^{M_\omega}$. As $j$ is 3-iterable and 
\begin{equation*}
 j_{m,\omega}(j_m) = j_{m,\omega}(j_{0,m}(j))=j_{0,\omega}(j)=j_\omega,
\end{equation*}
everything works. 
	
Consider a branch of $T_3$ in $M_{\omega\cdot 3}$. 
Then, calling $\gamma_\omega=\sup_{n\in\omega}\gamma_n$, $\eta_\omega=\sup_{n\in\omega}\eta_n$, $\delta_\omega=\sup_{n\in\omega}\delta_n$, $e=\bigcup_{n\in\omega}e^n$, $k_m=\bigcup_{n\in\omega}k^n_m$ for any $m\in\omega$, $h=\bigcup_{n\in\omega}h^n$, $g_m=\bigcup_{n\in\omega}g^n_m$ for any $m\in\omega$, by the previous results we have that $e \colon V_{\gamma_\omega}\prec V_{\gamma_\omega}$, $k_m \colon V_{\gamma_\omega}\prec (V_{\eta_\omega})^{M_\omega}$ for any $m\in\omega$, $h \colon (V_{\eta_\omega})^{M_\omega}\prec (V_{\eta_\omega})^{M_\omega}$ and $g_m \colon (V_{\eta_\omega})^{M_{\omega\cdot 2}}\prec (V_{\delta_\omega})^{M_{\omega\cdot 2}}$. 
As before, by the regularity of $\lambda$, $\lambda_\omega$ and $\lambda_{\omega\cdot 2}$ in $M_{\omega\cdot 3}$, $\gamma_\omega<\lambda$, $\eta_\omega<\lambda_\omega$ and $\delta_\omega<\lambda_{\omega\cdot 2}$. 
Moreover, $e$ and $h$ are 1-iterable, and point (4) of the definition of $T_3$ guarantees that $k_m^+(e_m)=h$ for any $m\in\omega$. 
We want to prove that $k_\omega^+(e_\omega)=h$. 
But $e_\omega=e_{m,\omega}^+(e_m)$ for all $m\in\omega$, so $k_{\omega}^+(e_{m,\omega}^+(e_m))$ is, by definition, $k_m^+(e_m)$.

Now, if $e_\omega=h$, then we have that the $\omega$-iterate of $e$ is 1-iterable, and so $e$ is 2-iterable. 
Otherwise this is the picture, where $(V_{\eta_\omega})^{M_\omega}=\bar{N}_0$ and $(V_{\delta_\omega})^{M_{\omega\cdot 2}}=\bar{\bar{N}}_0$:
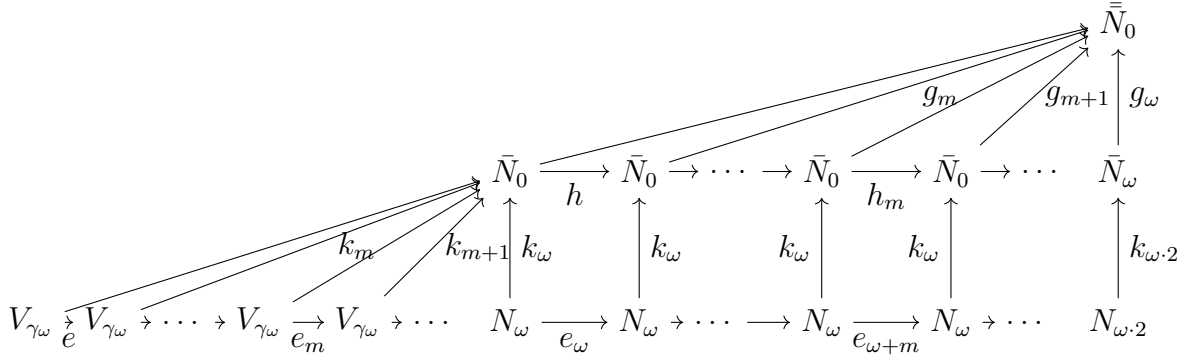
\[
\begin{tikzpicture}
 \node (M0) {$V_{\gamma_\omega}$};
 \node (M1)[right of = M0] {$V_{\gamma_\omega}$}
 edge [<-] node[below]{$e$}(M0);
 \node (dots)[right of = M1] {$\dots$}
 edge [<-] (M1);
 \node (Mn)[right of = dots] {$V_{\gamma_\omega}$}
 edge [<-] (dots);
 \node (Mn1)[right of = Mn, xshift=3mm] {$V_{\gamma_\omega}$}
 edge [<-] node[below]{$e_m$}(Mn);
 \node (dots2)[right of = Mn1] {$\dots$}
 edge [<-] (Mn1);
 \node (Nomega)[right of = dots2] {$N_\omega$};
 \node (upper)[above of = Nomega, yshift=1cm] {$\bar{N}_0$}
 edge [<-] node[right]{$k_\omega$}(Nomega)
	 edge [<-] (M0)
	 edge [<-] (M1)
	 edge [<-] node[left]{$k_m$}(Mn)
	 edge [<-] node[right]{$k_{m+1}$}(Mn1);
	\node (upper2)[right of = upper, xshift=7mm] {$\bar{N}_0$}
	 edge [<-] node[below]{$h$}(upper);
	\node (dots3)[right of = upper2, xshift=2mm] {$\dots$}
	 edge [<-] (upper2);
	\node (uppern)[right of = dots3, xshift=2mm] {$\bar{N}_0$}
	 edge [<-] (dots3);
	\node (uppern1)[right of = uppern, xshift=7mm] {$\bar{N}_0$}
	 edge [<-] node[below]{$h_m$}(uppern);
	\node (dots4)[right of = uppern1, xshift=2mm] {$\dots$}
	 edge [<-] (uppern1);
	\node (barNomega)[right of = dots4] {$\bar{N}_\omega$};
	\node (upperupper)[above of = barNomega, yshift=1cm] {$\bar{\bar{N}}_0$}
	 edge [<-] node[right]{$g_\omega$}(barNomega)
	 edge [<-] (upper)
	 edge [<-] (upper2)
	 edge [<-] node[left]{$g_m$}(uppern)
	 edge [<-] node[right]{$g_{m+1}$}(uppern1);
	\node (Nomega1)[right of = Nomega, xshift=7mm] {$N_\omega$}
	 edge [<-] node[below]{$e_\omega$}(Nomega)
	 edge [->] node[right]{$k_\omega$}(upper2);
	\node (dots5)[right of = Nomega1] {$\dots$}
	 edge [<-] (Nomega1);
	\node (Nomegan)[below of = uppern, yshift=-1cm] {$N_\omega$}
	 edge [<-] (dots5)
	 edge [->] node[left]{$k_\omega$}(uppern);
	\node (Nomegan1)[below of = uppern1, yshift=-1cm] {$N_\omega$}
	 edge [<-] node[below]{$e_{\omega+m}$}(Nomegan)
	 edge [->] node[left]{$k_\omega$}(uppern1);
	\node (dots6)[right of = Nomegan1] {$\dots$}
	 edge [<-] (Nomegan1);
	\node (Nomegaomega)[below of = barNomega, yshift=-1cm] {$N_{\omega\cdot 2}$}
	 edge [->] node[right]{$k_{\omega\cdot 2}$}(barNomega);
\end{tikzpicture} 
\]
By the usual reasoning, there exists $k_{\omega\cdot 2} \colon N_{\omega\cdot 2}\prec \bar{N}_{\omega}$.
We know that $(V_{\delta_\omega})^{M_{\omega\cdot 2}}$ is well-founded because $j$ is 3-iterable, but then by elementarity of $g_\omega$ we have that $\bar{N}_\omega$ is well-founded, and so by elementarity of $k_{\omega\cdot 2}$ also $N_{\omega\cdot 2}$ is well-founded.

Note that the usual remark on the cofinality of the possible $\lambda'$ under $\lambda$ still holds, with the same proof.
\end{proof}
	
The techniques used for $T_3$ can be used now to define $T_4$, $T_5$, and so on, therefore proving the first part of the theorem for $\alpha$ finite. 
For example this is the diagram generated by a branch of $T_4$.
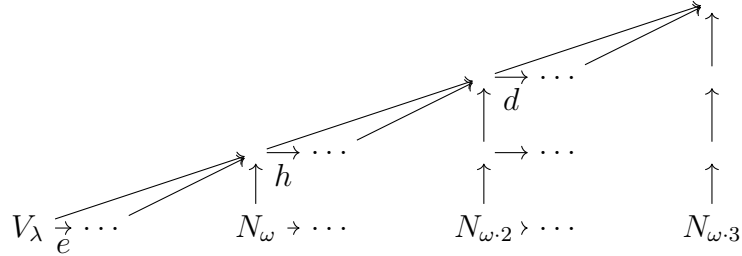
\[
	\begin{tikzpicture}
	 \node (base) {$V_\lambda$};
	 \node (dots) [right of = base] {$\dots$}
	 edge [<-] node[below]{$e$}(base);
	 \node (omega)[right of = dots, xshift=1cm]{$N_\omega$};
	 \node (u0) [above of = omega] {}
	 edge [<-] (base)
		edge [<-] (dots)
		edge [<-] (omega);
	 \node (dots2) [right of = omega] {$\dots$}
	 edge [<-] (omega);
	 \node (udots) [right of = u0] {$\dots$}
	 edge [<-] node[below]{$h$}(u0);
	 \node (omega2)[right of = dots2, xshift=1cm]{$N_{\omega\cdot 2}$};
	 \node (u1) [above of = omega2] {}
	 edge [<-] (omega2);
	 \node (uu0) [above of = u1] {}
		edge [<-] (u0)
		edge [<-] (udots)
		edge [<-] (u1);
	 \node (dots3) [right of = omega2]{$\dots$}
	 edge [<-] (omega2);
	 \node (udots2) [right of = u1]{$\dots$}
		edge [<-] (u1);
	 \node (uudots) [right of = uu0]{$\dots$}
		edge [<-] node[below]{$d$}(uu0);
	 \node (omega3)[right of = dots3, xshift=1cm]{$N_{\omega\cdot 3}$};
	 \node (u2) [above of = omega3] {}
	 edge [<-] (omega3);
	 \node (uu1) [above of = u2] {}
		edge [<-] (u2);
	 \node (uuu0) [above of = uu1] {}
		edge [<-] (uu0)
		edge [<-] (uudots)
		edge [<-] (uu1);
	\end{tikzpicture}
\]	
It is immediate now to see that if for all $n\in\omega$ there is a $j \colon V_\lambda\prec V_\lambda$ that has $M_{\omega\cdot n}$ well-founded, in particular there is a $j$ that is $(n+1)$-iterable, and this cofinally reflects the existence of an $e$ that is $n$-iterable.
	
We postpone the proof that ${<}\omega$-iterability cofinally strongly implies the existence of $n$-iterable embeddings for any $n\in\omega$ to the end of the proof of the theorem because it uses different techniques.
	
\begin{claim}
If there exists $j \colon V_\lambda\prec V_\lambda$ that is $\omega$-iterable (therefore $M_{\omega\cdot\omega}$ is well-founded), then there exists $\lambda'<\lambda$ and $e \colon V_{\lambda'}\prec V_{\lambda'}$ that is $n$-iterable for any $n\in\omega$. 
Moreover, the set of such $\lambda'$ is cofinal in $\lambda$.
\end{claim}
	
\begin{proof}
We want to build in $M_{\omega\cdot\omega}$ a tree $T_\omega$ that ``glues'' together all the $T_n$ trees, so that its branches will generate for any $n\in\omega$ a 1-iterable embedding $e^n$ in $M_{\omega\cdot n}$ and families of $k^n_m$ that witness that $e^n$ is 1-iterable and such that $(k^n_\omega)^+(e^n)=e^{n+1}$.
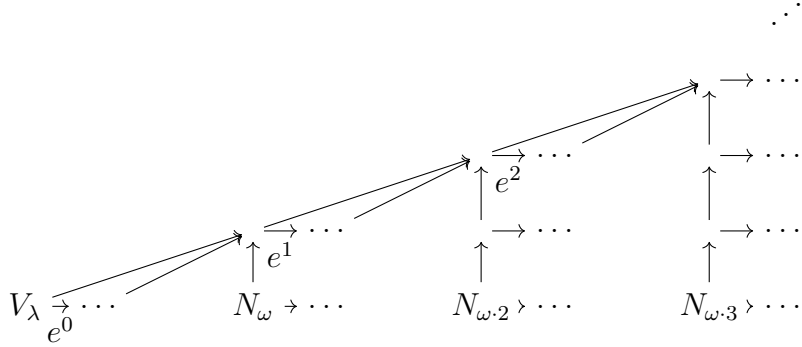
\[
	\begin{tikzpicture}
	 \node (base) {$V_\lambda$};
	 \node (dots) [right of = base] {$\dots$}
	 edge [<-] node[below]{$e^0$}(base);
	 \node (omega)[right of = dots, xshift=1cm]{$N_\omega$};
	 \node (u0) [above of = omega] {}
	 edge [<-] (base)
		edge [<-] (dots)
		edge [<-] (omega);
	 \node (dots2) [right of = omega] {$\dots$}
	 edge [<-] (omega);
	 \node (udots) [right of = u0] {$\dots$}
	 edge [<-] node[below]{$e^1$}(u0);
	 \node (omega2)[right of = dots2, xshift=1cm]{$N_{\omega\cdot 2}$};
	 \node (u1) [above of = omega2] {}
	 edge [<-] (omega2);
	 \node (uu0) [above of = u1] {}
		edge [<-] (u0)
		edge [<-] (udots)
		edge [<-] (u1);
	 \node (dots3) [right of = omega2]{$\dots$}
	 edge [<-] (omega2);
	 \node (udots2) [right of = u1]{$\dots$}
		edge [<-] (u1);
	 \node (uudots) [right of = uu0]{$\dots$}
		edge [<-] node[below]{$e^2$}(uu0);
	 \node (omega3)[right of = dots3, xshift=1cm]{$N_{\omega\cdot 3}$};
	 \node (u2) [above of = omega3] {}
	 edge [<-] (omega3);
	 \node (uu1) [above of = u2] {}
		edge [<-] (u2);
	 \node (uuu0) [above of = uu1] {}
		edge [<-] (uu0)
		edge [<-] (uudots)
		edge [<-] (uu1);
	 \node (dots4) [right of = omega3] {$\dots$}
		edge [<-] (omega3);
	 \node (udots3) [right of = u2] {$\dots$}
		edge [<-] (u2);
	 \node (uudots2) [right of = uu1] {$\dots$}
		edge [<-] (uu1);
	 \node (uuudots) [right of = uuu0] {$\dots$}
		edge [<-] (uuu0);
	 \node (fdots) [above of = uuudots] {$\iddots$};
	\end{tikzpicture}
\]	
It is important to define $T_\omega$ so that it is in $M_{\omega\cdot\omega}$, so that the argument of the absoluteness of well-foundedness gives a branch that is in $M_{\omega\cdot\omega}$, and therefore bounded below $\lambda$, so $T_\omega$ should be a subset of a set in $M_{\omega\cdot\omega}$. 
Note that $(V_{\kappa_{\omega\cdot\omega}})^{M_{\omega\cdot\omega}}=\bigcup_{n\in\omega}M_{\omega\cdot n}$, because of the properties of the direct limit and because $\crt(j_{\omega\cdot n,\omega\cdot\omega})=\crt(j_{\omega\cdot n})=\kappa_{\omega\cdot n}=\lambda_{\omega\cdot(n-1)}$ for any $n>0$. 
	
The most immediate approach would be to build all the embeddings $e^n$ and $k^n_m$ at the same time, step by step. 
At every passage, each approximation of the $e^n$ and $k^n_m$ will be actually in $M_{\omega\cdot n}$, as we have seen in the previous claims, therefore they would be all in $M_{\omega\cdot\omega}$. 
This approach, however, has a problem: in the finite cases, each node is a finite sequence of finite sequences, so it suffices to know that all its elements are in $M_{\omega\cdot n}$ to say that the whole tree is contained in it. 
In a $T_\omega$ defined in such a way we have instead infinite sequences, for example the first step would be to decide the critical points of all the $e^n$, and it is not clear why this sequence should be in $M_{\omega\cdot\omega}$. 
If we restrict ourselves only to the sequences that are in $M_{\omega\cdot\omega}$ then we are in trouble, as we possibly cannot then build a branch in $V$: for example the ``natural'' branch generated by $j$ will have the following first element:
\begin{equation*}
\langle \kappa_0,\kappa_\omega,\kappa_{\omega+\omega},\dots\rangle.
\end{equation*}
and this cannot be in $M_{\omega\cdot\omega}$, as the supremum of it is exactly $\kappa_{\omega\cdot\omega}$, that is regular in $M_{\omega\cdot\omega}$. 
The solution is to rearrange the pace at which the approximations are introduced, so that every sequence that appears in the revised tree $T_\omega$ is finite. 
We leave the details to the reader. 
Now the tree $T_\omega$ is on $M_{\omega\cdot\omega}$, and the proof is as before, cofinally strong implication included.
\end{proof}
	
As we have sufficiently analyzed the case of successor and limit of successors, the techniques just presented are enough to go up the hierarchy of the countable ordinals. 
The following claim completes the proof: 
	
\begin{claim}
Let $\nu\leq\omega_1$ be a limit ordinal. 
If there exists a $j \colon V_\lambda\prec V_\lambda$ that is $\alpha$-iterable for any $\alpha<\nu$, then there exists a $\lambda'<\lambda$ such that for any $\alpha<\nu$ there is an $e \colon V_{\lambda'}\prec V_{\lambda'}$ that is $\alpha$-iterable. 
\end{claim}

\begin{proof}
Let $j \colon V_\lambda\prec V_\lambda$ that is $\alpha$-iterable for any $\alpha<\nu$. 
Then, by the claims above, for any $\alpha<\nu$ there is a $\lambda'<\lambda$ and an $e \colon V_{\lambda'}\prec V_{\lambda'}$ that is $\alpha$-iterable. 
We want to prove that there is a single $\lambda'<\lambda$ that works for all the $\alpha<\nu$.
	
For any $\omega\cdot\alpha<\nu$ let 
\begin{equation*}
E_\alpha=\{\lambda'<\lambda : \exists k \colon V_{\lambda'}\prec V_{\lambda'}\text{ that is }\alpha\text{-iterable}\}. 
\end{equation*}
Then $E_\alpha\neq\emptyset$, and we need to prove that $\bigcap_{\omega\cdot\alpha<\nu}E_\alpha\neq\emptyset$. 
Since all the strong implications above are cofinal, $E_\alpha$ is cofinal in $\lambda$. 
We want to prove that $E_\alpha$ is definable in $V_\lambda$ using only $\alpha$ as a parameter. 
If $\lambda'\in E_\alpha$, let $e$ witness that, and let $n\in\omega$ be such that $e\in V_{\kappa_n}$. 
Then $N_\omega$, its $\omega$-iterate, is the set of the $e_{m,\omega}(y)$ with $m\in\omega$ and $y\in V_{\lambda'}$, therefore $|N_\omega|=|V_{\lambda'}|$. 
By induction, $|N_{\omega\cdot\beta}|=|V_{\lambda'}|$ for any $\beta\leq\alpha$, and as $N_{\omega\cdot\beta}$ is transitive this means that $N_{\omega\cdot\beta}\in V_{\kappa_n}$. 
Therefore $e$ is $\alpha$-iterable iff $V_\lambda\vDash e$ is $\alpha$-iterable, i.e., $V_\lambda$ computes correctly the iterability of embeddings inside it. 
Therefore $E_\alpha$ is definable in $V_\lambda$, so $j^+(E_\alpha)=E_\alpha$, and if $\eta<\crt(j)$ and $\beta$ is the $\eta$-th element of $E_\alpha$, then $j(\beta)$ is the $\eta$-th element of $E_\alpha$, i.e., $\beta$, so $\beta<\crt(j)$. 
But then the ordertype of $E_\alpha$ must be bigger than $\crt(j)$, otherwise $E_\alpha$ would be all inside $\crt(j)$ and not cofinal, and the first $\crt(j)$ elements of $E_\alpha$ are smaller than $\crt(j)$. 
This means that $\bigcup_{\alpha<\beta}(E_\alpha\cap\crt(j))\neq\emptyset$, as $\crt(j)$ is regular, and the proposition is proved. 
As $E_\alpha$ is cofinal under $\kappa_0$, by elementarity it is cofinal also under $\lambda$, therefore also cofinal strong implication is proved.
	\end{proof}
\end{proof}


\begin{thebibliography}{1}
 \bibitem{Dimonte2}
	 V. Dimonte, \emph{Totally non-proper ordinals beyond $L(V_{\lambda+1})$}, Archive for Mathematical Logic \textbf{50} (2011) 565--584
	
	\bibitem{Dimonte3}
	 V. Dimonte, \emph{A Partially Non-Proper Ordinal}, Annals of Pure and Applied Logic \textbf{163} (2012) 1309--1321.
 
 \bibitem{Dimonte}
	 V. Dimonte, \emph{I0 and rank-into-rank axioms}, Bollettino dell'Unione Matematica Italiana, (2017) https://doi.org/10.1007/s40574-017-0136-y
	
	\bibitem{Jech}
	 T. Jech, \emph{Set theory. The third millennium edition, revised and expanded}. Springer Monographs in Mathematics, Springer-Verlag, Berlin, 2003.
	
	\bibitem{Kanamori}
 A. Kanamori, \emph{The higher infinite. Large cardinals in set theory from their beginnings. Second edition}. Springer Monographs in Mathematics, Springer-Verlag, Berlin, 2003.

 \bibitem{Kunen}
 K. Kunen, \emph{Elementary embeddings and infinite combinatorics}. Journal of Symbolic Logic \textbf{36} (1971), 209--231.
	
	 \bibitem{Laver2}
	 R. Laver, \emph{On the algebra of elementary embeddings of a rank into itself}. Advances in Mathematics \textbf{110} (1995), 334--346.	
	
 \bibitem{Laver}
	 R. Laver, \emph{Implications between strong large cardinal axioms}. Annals of Pure and Applied Logic \textbf{90} (1997), 79--90.
	
	 \bibitem{Martin}
	 D.A. Martin, \emph{Infinite games}. Proceedings of the ICM, Helsinki, 1978, Acad. Sci. Fennica, Helsinki, 269--273.
	
	 \bibitem{MOF}
	https://mathoverflow.net/questions/191742/a-question-on-rank-to-rank-embeddings
	
	 \bibitem{Suzuki}
	 A. Suzuki, \emph{No elementary embedding from $V$ into $V$ is definable from parameters}. Journal of Symbolic Logic \textbf{64} (1999), 1591--1594.
	
	\bibitem{Woodin}
	W. H. Woodin, \emph{Suitable extender models II: beyond $\omega$-huge}. Journal of Mathematical Logic \textbf{11} (2011), 115--436.

 \end{thebibliography}
\end{document}